\documentclass[11pt,reqno]{amsproc}

\usepackage[margin=1in]{geometry}

\usepackage[usenames,dvipsnames]{color}

\usepackage[colorlinks=true, pdfstartview=FitV, linkcolor=blue,citecolor=blue, urlcolor=blue]{hyperref}

\title[Unique Ergodicity for  Fractionally Dissipative 2D Euler]{Unique Ergodicity for Fractionally Dissipated, Stochastically Forced 2D Euler Equations \footnote{ {\em Date:} \today} }

\author{Peter Constantin}
\address{Department of Mathematics, Princeton University, Princeton, NJ 08544}
\email{const@math.princeton.edu}

\author{Nathan Glatt-Holtz}
\address{Institute for Mathematics and Applications, University of Minnesota, Minneapolis, MN 55455}
\email{negh@ima.umn.edu}

\author{Vlad Vicol}
\address{Department of Mathematics, Princeton University, Princeton, NJ 08544}
\email{vvicol@math.princeton.edu}

\usepackage{amsmath, amsthm, amssymb,amsfonts,latexsym,verbatim,amsbsy,bookmark,times}

\usepackage[usenames,dvipsnames]{color}

\theoremstyle{plain}
\newtheorem{theorem}{Theorem}[section]
\newtheorem*{theorem*}{Theorem}
\newtheorem{definition}[theorem]{Definition}
\newtheorem{lemma}[theorem]{Lemma}
\newtheorem{proposition}[theorem]{Proposition}

\theoremstyle{definition}
\newtheorem{remark}[theorem]{Remark}

\newcommand{\LinI}{\mathcal{J}}
\newcommand{\Kon}{\mathcal{K}}
\newcommand{\LinN}{\mathcal{A}}
\newcommand{\DM}{\mathfrak{D}}
\newcommand{\DDM}{\mathbb{D}}
\newcommand{\ddt}{\frac{d}{dt}}

\renewcommand{\tilde}{\widetilde}

\def\BB{{\mathcal B}}
\def\PP{{\mathcal P}}
\def\ZZ{\mathbb{Z}}
\def\RR{\mathbb R}
\def\LL{{\mathbb L}}
\def\HH{{\mathbb H}}

\newcommand{\pd}[1]{\partial_{#1}}
\newcommand{\indFn}[1]{1 \! \! 1_{#1}}
\newcommand{\E}{\mathbb{E}}
\newcommand{\Prb}{\mathbb{P}}
\def\ZZF{ {\mathcal Z}}
\def\intint{\int\!\!\!\int}

\newcommand{\Vort}{\omega}

\newcommand{\bfU}{\boldsymbol{u}}
\newcommand{\TT}{\mathbb{T}^2}
\newcommand{\II}{\mathcal{I}}

\renewcommand{\phi}{\varphi}
\newcommand\eps{\varepsilon}

\newcommand\poly{ {\mathcal P} }
\newcommand\expo{ {\mathcal E} }

\begin{document}


\begin{abstract}
We establish the existence and uniqueness of an ergodic invariant measure for 
2D fractionally dissipated stochastic Euler equations on the periodic box, for \emph{any} power of the dissipation term.
\end{abstract}

\subjclass[2010]{}

\maketitle

\section{Introduction}
\label{sec:Intro}

Because of the combined effects of rapid rotation and small aspect ratio,
much of large scale atmospheric turbulence is dominated by two dimensional
dynamics. In this setting, the role of molecular dissipation is negligible, but other forms of dissipation do exist~\cite{HeldSQG95,HeldEtAll00,PauluisHeld02,HeldEtAll02}. Two dimensional turbulence has been extensively studied theoretically~\cite{Kraichnan67, Batchelor69,RobertSommeria92,Falkovich94,Constantin97,FoiasJollyManleyRosa02}, experimentally~\cite{Tabeling97,Tabeling98,Bodenschatz02,Swinney02} and numerically~\cite{PumirEtAl00,CastiglionePumir01}. See also the reviews~\cite{Frisch95,Tabeling02} and references therein. In such forced dissipative systems a common approach, both numerically and theoretically, is to use a frequency-localized stationary gaussian (white in time) stochastic process as a proxy for ``generic'' energy injection, see e.g.~\cite{Novikov1965,FrischSulem84,Eyink96,CIME08,BouchetSimonet09,KuksinShirikian12}.

The simplest form of dissipation, wave-number independent friction, leads to the damped-driven Euler equations~\cite{Benard00,ConstantinRamos07}. 
Unfortunately, the ergodic theory for the stochastically forced damped-driven Euler equations seems to be far from reach at the moment: this is in part due to the lack of compactness or continuous dependence in a suitable Polish space.  The natural space for 2D Euler, $L^{\infty} \cap L^1$ vorticity, is ill-suited to study the ergodic theory for SPDEs in the Markovian framework with the existing tools.
 
Weak wave-number dependence in the dissipation can be viewed as remedy for the difficulties encountered with damped and driven Euler equations. Our goal here is to address the question of what is the lowest power of dissipation in the fractionally dissipated Euler equations that allows the development of a rigorous ergodic theory. In the case of very weak wave-number dependence this question turns out to be quite non-trivial. Recently, the use of fractional dissipation as a regularizing term in models arising in fluid mechanics has become quite common, see e.g.~\cite{ConstantinCordobaWu01,Wu02b,KiselevNazarovVolberg07,CastroCordobaGancedoOrive09,CaffarelliVasseur10,Kiselev10,CaffarelliChanVasseur11,HKR11,ConstantinVicol12,ChaeWu12} and references therein. 

In this work we establish the existence and uniqueness of an ergodic invariant measure for the fractionally dissipative 2D Euler equation in vorticity form
\begin{align}
  &d \Vort + (\Lambda^\gamma \Vort + \bfU \cdot \nabla \Vort)dt = \sigma dW, \notag\\
  &\bfU = K \ast \Vort, \label{eq:frac:SNSE}\\
  & \Vort(0) = \Vort_0,\notag
\end{align}
where $\Lambda^\gamma = (- \Delta)^{\gamma/2}$ is the fractional Laplacian, and $\gamma$ is allowed to take {\em any} value in $(0,2]$.
Here $K$ is the Biot-Savart kernel, so that $\nabla^\perp \cdot \bfU = \Vort$ and $\nabla \cdot \bfU = 0$. The equations evolve on the periodic box $\TT = [-\pi , \pi]^2$.  The noise is white in time, colored in space, and degenerate, in the sense that it is supported on only finitely many Fourier modes.
This work is part of a larger goal to understand inviscid limits for weakly dissipated, stochastically forced Euler equations and related systems, in the class of invariant measures.  See also \cite{EKMS00, Kuksin2004, ConstantinRamos07, BouchetSimonet09, Kupiainen10, KuksinShirikian12, GlattHoltzSverakVicol2013} and containing references.

There exists a fairly well-developed ergodic theory of the stochastic Navier-Stokes equations in two dimensions.  As far as we know, the study of stochastic Navier-Stokes goes back to the 1960's~\cite{Novikov1965}, with the rigorous mathematical framework initially developed by~\cite{BensoussanTemam1972,VishikKomechFusikov1979,Cruzeiro1}. The ergodic theory for 2D stochastic NSE, and other nonlinear SPDEs, was initiated by \cite{FlandoliMaslowski1, ZabczykDaPrato1996} around the Doob-Khasminskii theorem.  
 This setting requires finite time smoothing of the Markov semigroup (the Strong Feller property) and a strong form of irreducibility.
 As such, these initial works required a very non-degenerate noise structure, that is stochastic forcing in all Fourier modes.
Following these pioneering works,  a number of authors have addressed the case of increasingly degenerate stochastic forcing~\cite{Ferrario99, Mattingly1,Mattingly2, BricmontKupiainenLefevere2001, EMattinglySinai2001, EchmannHairer01, KuksinShirikyan1,KuksinShirikyan2, MasmoudiYoung02,Mattingly03}. These authors realized the essential role played by Foias-Prodi-type estimates (determining modes)~\cite{FoaisProdi67} for obtaining ergodicity in nonlinear SPDEs. More recently, in a series of works \cite{MattinglyPardoux1, HairerMattingly06, HairerMattingly2008,HairerMattingly2011}
the unique ergodicity and mixing properties of the stochastic Navier-Stokes equations have been established for a class of very
degenerate (hypoelliptic) stochastic forcings.  In particular these authors introduced a notion of time asymptotic smoothing for 
the Markov semigroup and connected this property with unique ergodicity.   We will make central use of this ``asymptotic strong
Feller'' property here. For further recent developments and background on the ergodic theory of nonlinear SPDEs we refer the reader
to \cite{Kupiainen10, Debussche2011a, KuksinShirikian12} and references therein. 
The time asymptotic and statistically stationary behavior of the damped stochastic Euler has been studied 
in e.g. \cite{BessaihFlandoli2000,BessaihFerrario12, GlattHoltzSverakVicol2013}, but only in the context of weak solutions both in the PDE and probabilistic sense, which is far from the Markovian framework used here.

If the dissipation's wave number dependence is strong enough, i.e. for 
$\gamma \in (1,2)$, the argument in \cite{HairerMattingly06} appears to go through without major new ideas. The reason is that $\Lambda^\gamma$ is smoothing by $\gamma$ derivatives, while $ u \cdot \nabla \Vort  =  \nabla \cdot (u \Vort)$ has a one derivative loss.  Moreover, at the technical level, when $\gamma \in (1,2)$ one can simply work with the phase space $L^2$, where we have existence, uniqueness, and continuous dependence on data for the SPDE. Therefore, one can show the Markovian semigroup is Feller, obtain the needed exponential moment bounds, and the asymptotic strong Feller property, all in the $L^2$ phase space.

On the other hand, the case $\gamma \in (0,1]$ is hard for the following reasons: It appears from the above naive derivative counting that the case $\gamma \in (0,1]$ requires a new idea in order to appeal to Foias-Prodi-type arguments. No continuous dependence on data in the $L^2$ phase space is available, and even uniqueness might fail in $L^2$ for \eqref{eq:frac:SNSE}, so the Markovian framework breaks down in this space. To make sure we have uniqueness in the SPDE, we work in the phase space $H^r$ with $r>2$. The essential challenge now is that there is no cancellation property for the nonlinear term in $H^r$, and so obtaining moments becomes a highly non-trivial task. Moments are used extensively throughout the analysis: usually polynomial moments (with at most linear time-growth) are used to obtain the existence of invariant measures, while exponential moments are used essentially in obtaining the uniqueness.  Furthermore, due to this lack of cancellation in $H^r$, even establishing the Feller property is not trivial.

Our new ideas, which allow us to overcome the above mentioned technical difficulties are as follows. We developed a way to use the inherent parabolic smoothing in the equation, combined with a stopping time argument, to obtain the Feller property. The smoothing also takes care (after an arbitrarily small transient smoothing time) of the problem with the naive derivative counting.  In order to obtain estimates on the gradient of the Markov semigroup, we need to make use of exponential moments. Since we do not have them in $H^r$, we first need to address the control equation for the velocity fields in $L^2$, which only require the available $L^p$ exponential moments. We then use the instantaneous smoothing and interpolation to obtain the needed decay estimates in $H^r$. Parabolic smoothing may thus be also viewed as a tool to bootstrap available moment information from $L^p$ to $H^r$. Our exponential moment estimates make use of a sharpening of a lower bound on the dissipation term in $L^{p}$ from \cite{CordobaCordoba2004}.  This appears to be new, and may also be of some independent interest.

\subsection{Mathematical Setting and Assumptions on the Noise}
\label{sec:Setting}
The size of the periodic domain is chosen so that for convenience the lowest eigenvalue of $-\Delta$, and hence of the fractional Laplacian $\Lambda^\gamma$, is equal to $1$.  
Lebesgue spaces are denoted as usual by $L^p$ with $1\leq p \leq \infty$. For any $s > 0$ we shall denote by $H^{s}$ as the Hilbert space of mean zero elements $\Vort \in L^{2}(\TT)$ such that $\| \Lambda^{s} \Vort\|^{2}_{L^{2}} 
< \infty$.

Throughout this manuscript  fix $r>2$. We will be considering pathwise, that is probabilistically strong, solutions. Hence we fix a stochastic basis ${\mathcal S} = ( \Omega, {\mathcal F},  \Prb, \{ {\mathcal F}_t \}_{t\geq 0}, W)$. 
The noise term in \eqref{eq:frac:SNSE} is given explicitly as 
\begin{align} 
\sigma dW = \sum_{k \in \ZZF} \sigma_{k}(x) dW^{k}(t)
\label{eq:sigma:Def}
\end{align}
where we denote the set of forced modes by $\ZZF \subset {\mathbb Z}_0^2$. 
We will use the notation that for $s \geq 0$
\begin{align*}
 \|\sigma\|_{\HH^s}^2 := \sum_{k \in \ZZF} \| \sigma_k \|_{H^s}^2  
\end{align*}
Similarly we adopt for any $p \geq 2$,
\begin{align*}
   \|\sigma\|_{\LL^p}^p =  \int_{\TT} \Bigl( \sum_{k \in \ZZF} | \sigma_{k}(x) |^{2} \Bigr)^{p/2} dx.
\end{align*}

We assume for simplicity that for $|\ZZF| < \infty$ and $\ZZF$ is symmetric with respect to $k \mapsto -k$.
Also, to simplify the exposition we  consider the following explicit structure\footnote{For Proposition~\ref{prop:well-posedness} and Theorems~\ref{thm:smoothing}--\ref{thm:existence} we in fact do not require any  other assumptions on the noise term, except for $\| \sigma \|_{\HH^{s}} < \infty$ with $s=r+2$. In fact, these results hold with more general, possibly state-dependent (multiplicative) noise.} for the $\sigma_k$. As in \cite{HairerMattingly06}, let  $\{ q_{k}\}_{k \in \ZZF}$ be a collection of non-zero real numbers and let $e_k(x) = \sin(k \cdot x)$ for $k \in \ZZF \cap \ZZ_+^2$ and $e_k(x) = \cos (k \cdot x)$ for $k \in \ZZF \cap \ZZ^2_-$, where $\ZZ^2_+ = \{ k =(k_1,k_2) \in \ZZ^2 \colon k_2 > 0, \mbox{ or } k_2 = 0 \mbox{ and } k_1 >0\}$, and $\ZZ^2_- = - \ZZ^2_+$. We then let 
\begin{align}
\sigma_k(x) = q_k e_k(x)
\label{eq:NoiseCond2}
\end{align}
for any $k \in \ZZF$. For many of the below results we need $\ZZF$ to contain a sufficiently large ball around the origin in $\ZZ^2$. That is, we assume there exists a sufficiently large  integer $N >0$ such that 
\begin{align*}
\{ k \in \ZZ^2 \colon 0 < |k| \leq N \} \subset \ZZF.
\end{align*}

\subsection{Well-posedness and Markovian Framework}
For initial data $\Vort_0 \in H^r$, it may be shown
that \eqref{eq:frac:SNSE} has a unique global in time probabilistically strong, i.e. pathwise, solution in $H^r$.   More precisely we have the following
well-posedness result:
\begin{proposition}[\bf Well-posedness]
\label{prop:well-posedness}
  Fix a stochastic basis ${\mathcal S} = ( \Omega, {\mathcal F},  \Prb, \{ {\mathcal F}_t \}_{t\geq 0}, W)$ and consider any 
  $r > 2, \gamma > 0$.  Suppose, for simplicity that $\sigma$ takes the form \eqref{eq:NoiseCond2}. Then, for any $\Vort_{0} \in H^{r}$ there exists a unique $\Vort = \Vort(t, \Vort_{0}) =  \Vort(t, \Vort_{0}, \sigma W)$ satisfying \eqref{eq:frac:SNSE},
  in the time integrated sense and with the regularity
  \begin{align*}
	\Vort \in L^{\infty}([0,\infty);H^{r}) \cap L^{2}_{loc}([0,\infty); H^{r+ \gamma/2})
  \end{align*} 
  almost surely.
\end{proposition}
By making use of the change of variable $\bar{\Vort} = \Vort - \sigma W$, this existence and uniqueness result can 
be established using standard methods in a similar fashion to the 2D Euler/Navier-Stokes equation.  
See e.g. \cite{ConstantinFoias1988,MajdaBertozzi2002}.  Note however 
that the local and global existence for strong, pathwise solutions of the 2D and 3D Euler equations has been treated in a much more general setting, with multiplicative noise and in the presence of boundaries in \cite{GlattHoltzVicol2011}, 
and see also the references therein.  See e.g. \cite{Rozovskii1990, ZabczykDaPrato1992, PrevotRockner2007} for more on the general well-possedness theory for SPDE.

On the other hand, the fractional dissipation term present in \eqref{eq:frac:SNSE} leads to smoothing properties which
we would not expect from the damped Euler equation.  We will show in Theorem~\ref{thm:smoothing} 
that solutions smooth to an arbitrary  degree regularity after an arbitrarily short time.  With this smoothing effect we show in Proposition~\ref{prop:cont:dep} that solutions dependent continuously on initial conditions in the $H^{r}$ topology.

With this basic well-posedness in hand we may associate Markov transition functions to \eqref{eq:frac:SNSE} by
defining
\begin{align}
  P_t(\Vort_0, A) = \Prb( \Vort(t, \Vort_0) \in A) \quad \mbox{ for any } t \geq 0, A \in \mathcal{B}(H^r).
  \label{eq:Markov:trans:fn}
\end{align}
This defines the Markovian semigroup, also denoted $\{P_t\}_{t \geq 0}$
\begin{align}
P_t \phi(\cdot) = \E \phi(\Vort(t,\cdot)) = \int_{H^r} \phi(\Vort_0) P_t(\cdot,d\Vort_0),
\quad \mbox{ for any } \phi \in \mathcal{M}_b(H^r),
\label{eq:Markov:semigroup}
\end{align}
where $\mathcal{M}_b(H^r)$ denotes the collection of bounded, real valued,
Borel measurable functions mapping from $H^r$.  We will denote $C_b(H^r)$ to be 
the collection of continuous real valued functions mapping from $H^r$.  We will show 
in Section~\ref{sec:Feller:existence} below, that $\{P_t\}_{t \geq 0}$ is \emph{Feller}
meaning that $P_t$ maps $C_b(H^r)$ into $C_b(H^r)$ for every $t \geq 0$.
Let us recall that for any Borel probability measure $\mu$, the dual semigroup $P_t^*$ acts as
\begin{align*}
P_t^\ast \mu (\cdot) = \int_{H^r} P_t(\Vort_0,\cdot) d \mu(\Vort_0).
\end{align*}
Note that $P_t^\ast$ may be defined to act on any finite signed Borel measure $\mu$.
Then $\mu \in \mbox{Pr}(H^r)$ is an {\em invariant measure} for $\{P_t\}_{t\geq 0}$ if $P_t^\ast \mu = \mu$ for all $t\geq 0$.

\subsection{Notation}
For the sake of readability, throughout this manuscript we shall adopt the following notational conventions. All constants are deterministic and independent of time.

\begin{itemize}
 \item[(i)] $C$ shall denote a sufficiently large positive constant that depends on  $r,\gamma$, and on the constants arising in the Sobolev, Poincar\'e, Burkholder-Davis-Gundy, and other inequalities. The value of $C$ may change from line to line. When the constant $C$ depends on other parameters $\lambda$, we shall explicitly remind the reader of this dependance by writing $C(\lambda)$.
 \item[(ii)] $\poly(x)$ shall denote a polynomial of the type $1 + x^{q}$, where the degree of the polynomial is suppressed in the notation. We shall also write $\poly(x,y)$ to denote a polynomial $1 + x^{q_{1}} + y^{q_2}$, where again the dependence on $q_1$ and $q_2$ is suppressed.
 \item[(iii)] $\expo(\kappa, x)$ shall denote the function $\exp(\kappa(1+x^2))$. Below, $\kappa$ shall always take the form $\kappa = 1/(C \poly(\| \sigma\|))$ for a suitable norm $\| \cdot \|$ of $\sigma$, which we will specify, and a universal constant $C$  as in (i) above.
\end{itemize}

\subsection{Main Results} 
\label{sec:Main:Results}
We now turn to describe the main results of the work, and to lay out some of the challenges  involved in their proofs.
As mentioned above, at the heart of our argument is obtaining moment bounds in high Sobolev spaces for solutions of \eqref{eq:frac:SNSE}. The first main result gives polynomial moment bounds for the $H^{r+\alpha(t)}$ norm of $\Vort(t)$, with $\alpha(t)$ increasing, and these bounds grow only linearly in time. This secular growth is in turn essential for the existence and uniqueness results below (cf.~Theorems~\ref{thm:existence} and \ref{thm:uniqueness}).

\begin{theorem}[\bf Polynomial Sobolev Moments and Smoothing]\label{thm:smoothing}
Fix $r > 2$, $ \gamma >0$, and let $m \geq 0$ and $T_m > 0$ be arbitrary.  Define
\begin{align} 
\alpha(t)  = 
\begin{cases}
m t T_m^{-1}, & t \in [0, T_m],\\
m, & t > T_{m}.
\end{cases}
\label{eq:alpha:def}
\end{align}
Then, for any $T > 0$ and any $q \geq 2$ we have
\begin{align}
& \E \left(\sup_{t \in [0,T]} \|\Lambda^{r+\alpha(t)} \Vort (t) \|_{L^2}^q + \int_0^T \|\Lambda^{r + \gamma/2 + \alpha(t)} \Vort(t)\|_{L^2}^2 \|\Lambda^{r+\alpha(t)} \Vort (t) \|_{L^2}^{q-2} dt\right)  \notag\\
&\qquad \leq C \PP(\E \|\Vort_0\|_{H^r}) + C T \PP(\|\sigma\|_{\HH^{r+m}}),
\label{eq:Smoothing:moments}
\end{align}
where $C= C(q,T_m, m)$ is a sufficiently large constant, independent of $T$. The polynomial $\PP$ is given explicitly in \eqref{eq:KB:moment:3} below.
\end{theorem}

We emphasize that $m$, the number of derivatives we want to gain, and $T_m$, the time in which this gain is achieved, can be taken arbitrarily large, respectively arbitrarily small. This is a quantitative control on the parabolic smoothing effects inherent in the equations. The techniques outlined in the proof of Theorem~\ref{thm:smoothing} below, combined with the arguments in~\cite{FoiasTemam1989,Mattingly2002}, may be used to show that in fact the equations lie in a Gevrey-class when $t>0$. 

The main difficulty in establishing Theorem~\ref{thm:smoothing} is that in high Sobolev spaces $H^s$ with $s>0$, unlike the case $s=0$, we do not have that $\int \Lambda^s (u \cdot \nabla \Vort) \Lambda^s \Vort dx = 0$. In order to obtain bounds that do not blow up in finite time, we use a commutator estimate which shows that the $H^s$ norm is under control globally, if the expected value of the $H^1$ norm to a large power is integrable in time. In turn, to obtain such polynomial moments for the $H^1$ norm, upon integration by parts, it is sufficient to obtain polynomial moment bounds for high $L^p$ norms of the solution. The latter is achieved using that the nonlinear term vanishes in $L^p$, and the positivity of the fractional Laplacian in $L^p$, see \cite{CordobaCordoba2004} and Appendix~\ref{app:Poincare} below. The above described argument of bootstrapping moments from $L^p$ to $H^1$ and then to $H^{r+\alpha(t)}$ is given in Subsection~\ref{sec:polyn:mom} below.

While the polynomial moments established in Theorem~\ref{thm:smoothing} are sufficient in order to establish the existence and regularity of invariant measures of \eqref{eq:frac:SNSE}, in order to establish gradient estimates for the Markov semigroup, which is an essential step for uniqueness of invariant measures, exponential moments are needed. 
As with the case of the Navier-Stokes equations ($\gamma=2$), classical arguments can be used to establish exponential moments for the $L^2$ norm of the solution.  
However in the Naver-Stokes case $L^2$ is also the phase space where the Markov semigroup evolves. In the fractional case considered here the Markov semigroup is evolving on $H^r$, 
and this discrepancy between the space where exponential moments are available and the phase space causes a number of difficulties.  At this stage, in order to be able to use the
parabolic smoothing property we make critical use of of exponential moments for \emph{any} large $L^p$ norm of which is the next result.

\begin{theorem}[\bf Exponential Lebesgue Moments]\label{thm:exponential}
Let $p \geq 2$ be even, and $T>0$ be arbitrary. There exists $\kappa_0 >0$ with
\[
\kappa_0 = \frac{1}{C(p) \PP(\|\sigma\|_{\LL^p})},
\] 
such that for every $\kappa \in (0,\kappa_0]$ we have
\begin{align} 
\E \exp\left(  \kappa   \| \Vort(T)\|_{L^{p}}^{2} \right) 
+\E \int_0^T  \exp\left(   \kappa   \| \Vort(t)\|_{L^{p}}^{2} \right) dt  \leq C \E \expo(\kappa,   \| \Vort_0\|_{L^{p}}^{2})  + C \kappa T 
\label{eq:expo:mom:I}
\end{align}
where $C=C(p)$. Moreover, for every $\kappa$ in this range we have that
\begin{align} 
\E \exp \left( \kappa   \int_0^T \|\Vort(s)\|_{L^p}^2 ds   \right) 
\leq  e^T \E \expo(\kappa^{1/2},  \|\Vort_0\|_{L^p}^2)
\label{eq:expo:mom:II}
\end{align}
holds.
\end{theorem}
We notice that the growth in time of \eqref{eq:expo:mom:I} is only linear, and the exponential growth in time of \eqref{eq:expo:mom:II} is at a rate that is independent of $\sigma$. The proof of Theorem~\ref{thm:exponential}, given in Subsection~\ref{sec:Vort:exp} below, is based on the It\=o Lemma in $L^p$~\cite{ZabczykDaPrato1992,Krylov2010}, and a Poincar\'e inequality in $L^p$ for fractional powers of the Laplacian, given in Proposition~\ref{prop:Poincare}. More precisely, for $p\geq 2$ even, we prove that 
\begin{align*}
\int_{{\mathbb{T}}^2} \theta^{p-1}(x) \Lambda^\gamma \theta(x) dx \geq \frac{1}{C_{\gamma}} \|\theta\|_{L^p}^p + \frac{1}{p} \| \Lambda^{\gamma/2}  ( \theta^{p/2}) \|_{L^2}^2 
\end{align*}
holds, with an explicit constant $C_{\gamma} \geq 1$ given by \eqref{eq:poincare:constant} below. 
The lower bound \eqref{eq:Poincare}, but without the $L^p$ norm of $\theta$ on the right side, was proven by C\'ordoba and C\'ordoba in~\cite{CordobaCordoba2004}.
Since $\theta^{p/2}$ is not of zero mean when $p\geq 4$ is even, this lower bound does not however follow directly from~\cite{CordobaCordoba2004}. Note that  in the case $\gamma=2$ we have a local operator, $-\Delta$, and the above estimate easily follows using integration by parts. In the fractional case, due to the lack of a Leibniz rule, we need a different argument, given in Appendix~\ref{app:Poincare} below.

In the next theorem we establish the Feller property for the Markov semigoup $P_t$ associated to \eqref{eq:frac:SNSE}, and prove the {\em existence} of an ergodic invariant measure for the dual semigroup, which additionally is supported on smooth functions.

\begin{theorem}[\bf Existence and Regularity of Invariant Measures]\label{thm:existence}
Let $r>2$ and $\gamma >0$. The system \eqref{eq:frac:SNSE} defines, for $t \geq 0$ 
a Feller semigroup $P_{t}$ on $H^{r}$.   There exists an ergodic invariant measure $\mu$ for $\{P_t\}_{t \geq 0}$. Moreover, 
every invariant measure $\mu$ of $P_{t}$ is supported on $C^\infty$ and 
\begin{align*} 
\int_{H^r} \|\Vort\|_{H^s}^q d\mu(\Vort) < \infty
\end{align*}
for every $s\geq r$, and for any $q \geq 2$.
\end{theorem}

In the classical case of the stochastic Navier-Stokes equations, the Feller property follows directly from a continuous dependence
estimate on data in the phase space $L^{2}$.   In the fractional case with $\gamma \ll 1$, we face two complicating 
factors. The standard continuous dependence on data estimates in $H^{r}$ do not appear to work since we cannot control stray terms arising from linearization, as can be seen from a naive accounting based on the number of derivatives (for the deterministic Euler equation this is in fact not true~\cite{Masmoudi07}). To overcome this difficulty we make careful use of a parabolic smoothing argument, by controlling the difference of two $H^{r+\alpha(t)}$ solutions in $H^{r-1+\alpha(t)}$. Coupled with the bounds available from  Theorem~\ref{thm:smoothing}, with $m=1$ and $T_m$ sufficiently small, this allows us to control the difference of the solutions in $H^r$, for any strictly positive time.  Even leaving the regularity issue aside,   in contrast to the classical case where the Feller property is an immediate ``pathwise'' inference from the Dominated Convergence Theorem, due lack of cancellations here we must invoke a delicate stopping time and density argument. This is the content of Proposition~\ref{prop:Feller} below. See also the recent work \cite{GHKVZ13} where a similar approach has been used to address multiplicative noise. 

With the Feller property now in hand the existence and regularity of invariant measure now follows from Theorem~\ref{thm:smoothing} with the aid of standard long-time averaging arguments.  While we only go so far as to give the details for the $C^{\infty}$ 
support of $\mu$ we believe it should be possible to show that $\mu$ is in fact supported on Gevrey-class functions.
It is also worth emphasizing that up to this point in the work our arguments extend trivially to any additive noise with
a sufficiently smooth $\sigma$ and
indeed even to certain classes state dependent noise structures.  The proof of Theorem~\ref{thm:existence} is given
at the end of Section~\ref{sec:feller}.

\begin{theorem}[\bf Uniqueness of Invariant Measures]\label{thm:uniqueness}
Let $r>2$ and $\gamma>0$. There exists an $N=N(\gamma, r, \|\sigma\|_{\LL^{6/\gamma}})$, such that if the ball of radius $N$ in $\ZZ_0^2$ lies inside $\ZZF$, then there exists a unique and ergodic  invariant measure.
\end{theorem}
The proof of Theorem~\ref{thm:uniqueness} is carried out in Section~\ref{sec:ASF} and consists of two principal steps.
First we establish a certain time-asymptotic smoothing property of the Markov semigroup associated to \eqref{eq:frac:SNSE}. 
More specifically, we establish that $P_t$ satisfies the so-called {\em asymptotically strong Feller property}, which was introduced in \cite{HairerMattingly06}, and is recalled here 
in Definition~\ref{def:ASF} below. In practice this is achieved through an estimate on the gradient of the Markov semigroup obtained in Proposition~\ref{thm:ASF:easy}. In this setting, using some tools 
from Malliavin calculus, the gradient estimate boils down to constructing a suitable ``control'', which assigns to every perturbation in the initial data a perturbation in the noise.  This perturbation in the chosen so that the global 
dynamics is controlled by the dynamics on a sufficiently large, but finite, number of determining Fourier modes (Foias-Prodi estimates). Asymptotically, as $t \to \infty$, one has to show that the size of this control vanishes, which encodes the time-asymptotic smoothing of the Markov semigroup. 

For $\gamma$ small the difficulty here is twofold, even in the so called ``essentially elliptic'' case, where we force all the determining modes of the system. First, even if this number of modes $N$ is sufficiently large, a simple derivative count shows that for the control to decay one has to use the parabolic smoothing described in Theorem~\ref{thm:smoothing}. This aside, a second difficulty arises: we do not have exponential moments for the $H^s$ norms of the solution when $s\geq 1$, and such exponential moments appear to play an indispensable role in such gradient estimates. We overcome this difficulty by first proving that the $H^{-1}$ norm of the control decays, which only requires exponential moments for the $L^p$ norm of the solution, available by Theorem~\ref{thm:exponential}. We then use the smoothing effects and interpolation to bootstrap this $H^{-1}$ decay to a decay in $H^r$.

The second step is to establish that $P_t$ is weakly irreducible at $0$, meaning that $0$ is in the support of every invariant measure. This property follows from uniform estimates on the stationary solution established after Proposition~\ref{prop:higher} and the following property of the equations: the unforced dynamics are driven to $0$, and moreover this fixed point is stable under perturbations in the forcing.
The precise estimates which lead to the weak irreducibility property are given in Section~\ref{sec:WI}. 

Combining these two main steps, the asymptotic strong Feller property and the weak irreducibility, we now rely on the following fundamental result:

{\bf Theorem} \cite[Theorem 3.16]{HairerMattingly06}
{\em Let $\mu$ and $\nu$ be two distinct ergodic invariant probability measures for $P_t$. If $P_t$ is asymptotically strong Feller at $\Vort$, then $\Vort \not \in {\rm supp} \mu \cap {\rm supp} \nu$.}

Using the above result, in view of Theorem~\ref{thm:existence} we may now infer the uniqueness of invariant measures.

\section{A Priori Estimates. Moment Bounds and Instantaneous Smoothing}
\label{sec:moments}
\setcounter{equation}{0}

In this section we establish the following moment bounds for solutions $\Vort(t,\Vort_{0})$ of \eqref{eq:frac:SNSE}. For the sake of generality, we consider possibly random initial data, but the moment bounds obtained here will be applied in forthcoming sections with deterministic initial conditions.

\subsection{Polynomial Moments and Smoothing Estimates}
\label{sec:polyn:mom}

The goal of this subsection is to prove Theorem~\ref{thm:smoothing}, which is achieved in several steps. The first step is to obtain an moment bound for $L^p$ norms of $\Vort$. The second step is to bootstrap using a commutator estimate and obtain polynomial moments for the $H^1$ norm of the vorticity. The last step is to use the inherent parabolic regularization in the equation to further bootstrap and prove polynomial moment bounds on high Sobolev norms. We emphasize that all the moment bounds obtained in this subsection grow at most linearly with time.   This is essential way to establish the existence of invariant measures below in Section~\ref{sec:Feller:existence}.

\subsubsection{Estimates for \texorpdfstring{${\Vort}$}{w} in \texorpdfstring{$L^{p}$}{Lp}, \texorpdfstring{$p \geq 2$}{p>2}.}
\label{sec:Vort:Lp}
We now prove moment bounds for $L^{p}$ norms of the solution, with $p \geq  2$, and {\em even}.  Applying the It\=o lemma pointiwse in $x$ and the stochastic Fubini theorem, we obtain the following $L^p$ version of the It\=o lemma (see also \cite{Krylov2010})
\begin{align}
d \| \Vort\|_{L^{p}}^{p} =&\left( -p  \int_{\TT}  \Vort^{p-1} \Lambda^\gamma \Vort dx + \frac{p(p-1)}{2} \sum_{l \in \ZZF}  \int_{\TT} \sigma_l^2  \Vort^{p-2} dx \right) dt + p \sum_{l\in \ZZF} \left( \int_{\TT} \sigma_l  \Vort^{p-1} dx \right) d W^l
\notag\\
=:& \left(-pT_{1,p} + \frac{p(p-1)}{2}T_{2,p}\right)dt + p\sum_{l \in \ZZF} S_{l,p} dW^l. \label{eq:pk:Ito}
\end{align}
Using Proposition~\ref{prop:Poincare}, we have a lower bound on the fractional Laplacian  
\begin{align}
p T_{1,p} = p  \int_{\TT} \Vort^{p -1} \Lambda^\gamma \Vort dx \geq \frac{1}{C_\gamma} \|\Vort\|_{L^p}^p +  \int_{\TT} \left| \Lambda^{\gamma/2} (\Vort^{p/2}) \right|^2 dx
\label{eq:CC:lower:bound}
\end{align}
where $C_\gamma \geq 1$ may be computed explicitly.
A standard H\"older and  $\epsilon$-Young bound for the second term on the right side of \eqref{eq:pk:Ito} yields
\begin{align} 
\frac{p(p-1)}{2} T_{2,p}  
\leq p^2 \|\sigma\|_{\LL^{p}}^2 \| \Vort\|_{L^{p}}^{p-2} \leq \frac{1}{2C_\gamma} \|\Vort\|_{L^{p}}^{p} + C  \|\sigma\|_{\LL^{p}}^{p} \label{eq:pk:sigma:upper}
\end{align}
where $C_\gamma$ is the constant appearing in \eqref{eq:CC:lower:bound}.
We integrate \eqref{eq:pk:Ito} on $[0,T]$, take expected values and use the estimate \eqref{eq:CC:lower:bound}--\eqref{eq:pk:sigma:upper} to arrive at
\begin{align} 
& \E \|\Vort(T)\|_{L^{p}}^{p} +  \int_0^T \E \| \Vort(s)\|_{L^{p}}^{p} ds  \leq C \E \|\Vort_0\|_{L^{p}}^{p}+ C T \| \sigma \|_{\LL^{p}}^{p} . \label{eq:pk:kth:moment}
\end{align}

\subsubsection{Estimates for \texorpdfstring{${\Vort}$}{w} in \texorpdfstring{$H^{1}$}{H1}}
\label{sec:Vort:H1}

We now obtain moment bounds for the $H^1$ norm of the vorticity. First we deal with quadratic moments by appealing to the It\=o lemma in $H^1$ 
\begin{align}
d \| \Vort\|_{H^1}^2 + 2 \| \Vort\|_{H^{1+\gamma/2}}^2 dt 
= \left( \|\sigma\|_{\HH^1}^2 - 2 \int_{\TT} u \cdot \nabla \Vort \Delta \Vort dx \right) dt 
+ 2  \sum_{l\in \ZZF} \left( \int_{\TT} \sigma_l \Delta \Vort   dx \right) d W^l.
\label{eq:Ito:H1}  
\end{align}
Upon integration by parts, using that $\nabla \cdot u = 0$, the nonlinear term may be bounded as
\begin{align}
\left| \int_{\TT} u \cdot \nabla \Vort \Delta \Vort dx \right| 
= \left| \int_{\TT} \nabla \Vort \colon \nabla u \cdot \nabla \Vort  dx \right| 
\leq \|\nabla \omega\|_{L^{2+\eps}}^2 \| \nabla u\|_{L^{\frac{2+\eps}{\eps}}}
\label{eq:KB:step2:non:1}
\end{align}
where $\eps = \eps(\gamma) := 2\gamma/(4-\gamma)$ is defined such that $H^{\gamma/4} \subset L^{2+\eps}$ by Sobolev embedding. Using that $\nabla u$ is given by a matrix of  Riesz transforms acting on $\Vort$, and the Gagliardo-Nirenberg estimate, the right side of \eqref{eq:KB:step2:non:1} is further bounded as
\begin{align}
\|\nabla \omega\|_{L^{2+\eps}}^2 \| \nabla u\|_{L^{\frac{2+\eps}{\eps}}} 
&\leq C  \| \Vort\|_{L^{\frac{2+\eps}{\eps}}} \| \Vort\|_{H^{1+\frac{\gamma}{4}}}^2   \leq C \|\Vort\|_{L^{\frac{4}{\gamma}}} \|\Vort\|_{L^2}^{\frac{\gamma}{2+\gamma}} \|\Vort\|_{H^{1+\frac{\gamma}{2}}}^{2- \frac{\gamma}{2+\gamma}} \notag\\
&\leq  \|\Vort\|_{H^{1+\frac{\gamma}{2}}}^{2} + C  \|\Vort\|_{L^2}^2 \|\Vort\|_{L^{\frac{4}{\gamma}}}^{\frac{2(2+\gamma)}{\gamma}}  \leq \|\Vort\|_{H^{1+\gamma/2}}^{2} + C \|\Vort\|_{L^{\frac{4}{\gamma}}}^{\frac{4}{\gamma} +4} 
\label{eq:KB:step2:non:2}
\end{align}
for some sufficiently large constant $C$ that depends on $\gamma \in (0,2]$ and the size of the periodic box. Letting 
\begin{align}
p_\gamma =  4 + \frac{4}{\gamma}, \label{eq:p:gamma}
\end{align}
and using once more the H\"older and Poincar\'e inequalities, \eqref{eq:Ito:H1} gives
\begin{align}
d \| \Vort\|_{H^1}^2 +   \| \Vort\|_{H^1}^2 dt 
\leq \left( \|\sigma\|_{\HH^1}^2 + C  \|\Vort\|_{L^{p_{\gamma}}}^{p_{\gamma}} \right) dt 
+ 2  \sum_{l\in \ZZF} \left( \int_{\TT} \sigma_l \Delta \Vort   dx \right) d W^l
\label{eq:Ito:H1:2}  
\end{align}
and hence, upon integrating on $[0,T]$ and taking expected values we arrive at
\begin{align}
\E \| \Vort(T)\|_{H^1}^2 +  \int_0^T \E \| \Vort(s)\|_{H^{1}}^2 ds
&\leq \E \| \Vort_0\|_{H^1}^2 + T \|\sigma\|_{\HH^1}^2 + C \int_0^T \E \|\Vort(s)\|_{L^{p_{\gamma}}}^{p_{\gamma}} ds \notag\\
&\leq \E \| \Vort_0\|_{H^1}^2 +  C  \E \|\Vort_0\|_{L^{p_{\gamma}}}^{p_{\gamma}} + C  T \left(\|\sigma\|_{\HH^1}^2 + \| \sigma \|_{\LL^{p_{\gamma}}}^{p_{\gamma}} \right).
\label{eq:KB:moment:2:1}  
\end{align}
In the last inequality above we have appealed to the $L^{p_{\gamma}}$ moment bound \eqref{eq:pk:kth:moment} above. 

In Section~\ref{sec:Vort:Hr} below we will in fact need bounds on $\E \|\Vort\|_{H^1}^{q}$, with $q \geq 2$ possibly large, depending on $\gamma$. To this end, we apply the It\=o formula to the function $\phi(x) = x^{q/2}$, and $x(t) = \|\Vort (t) \|_{H^1}^2$ and obtain as in \eqref{eq:Ito:H1:2} that
\begin{align}
d \|\Vort\|_{H^1}^q 
&\leq \frac{q}{2} \|\Vort\|_{H^1}^{q-2} \left( -  \|\Vort\|_{H^1}^2 + \|\sigma\|_{\HH^1}^2 + C \|\Vort\|_{L^{p_{\gamma}}}^{p_{\gamma}}\right) dt\notag\\
&\qquad + q \|\Vort\|_{H^1}^{q-2} \sum_{l\in\ZZF} \left( \int_{\TT} \sigma_l \Delta \Vort   dx \right) d W^l  + \frac{q (q-2)}{2} \|\sigma\|_{\HH^1}^2 \|\Vort\|_{H^1}^{q-2} dt \label{eq:Ito:H1:3}
\end{align} 
for any $q\geq 2$. Integrating \eqref{eq:Ito:H1:3} on $[0,T]$, taking expected values and using the $\eps$-Young inequality three times, we arrive at
\begin{align} 
\E \|\Vort(T)\|_{H^1}^q + \frac{q}{8} \int_0^T \E \| \Vort(s)\|_{H^1}^q ds & \leq \E \|\Vort_0\|_{H^1}^q + C T \|\sigma\|_{\HH^1}^q  + C \int_0^T \E \|\Vort(s)\|_{L^{p_\gamma}}^{\frac{q p_\gamma}{2}} ds  \notag\\
& \leq \E \|\Vort_0\|_{H^1}^q + C \E \|\Vort_0\|_{L^{qp_\gamma/2}}^{qp_\gamma/2} + C T \left( \|\sigma\|_{\HH^1}^q  +   \| \sigma \|_{\LL^{q p_\gamma/2}}^{ q p_\gamma/2}  \right) \notag\\
& \leq C \poly(\E\|\Vort_0\|_{H^1}) + C T \poly ( \|\sigma\|_{\HH^1} ) 
\label{eq:KB:moment:2}
\end{align}
where in the second inequality we have used the H\"older inequality, and the moment bound \eqref{eq:pk:kth:moment} for the $L^{q p_\gamma/2}$ norm. Here $C$ depends on the parameter $q$.

\subsubsection{Estimates for \texorpdfstring{$\Vort$}{w} in \texorpdfstring{$H^{r +\alpha(t)}$}{H r+alpha}, smoothing.}
\label{sec:Vort:Hr}
In this final step of the proof of Theorem~\ref{thm:smoothing}, we make estimates on the $L^2$ norm of $\Lambda^{r + \alpha(t)} \Vort(t)$. Here 
for the sake of brevity we define
\begin{align*} 
s(t) = r + \alpha(t)
\end{align*}
and note that ${s}(t)>2$ for all $t\geq 0$.  
The It\=o lemma in $L^2$, applied to $\Lambda^{s(t)} \Vort$, yields
\begin{align}
&d \| \Lambda^{{s}} \Vort\|_{L^2}^2 + 2 \|\Lambda^{{s}+\gamma/2} \Vort\|_{L^2}^2 dt - 2 \dot{\alpha}(t) \|  \Lambda^{{s}} (\log \Lambda)^{1/2} \Vort \|_{L^{2}}^{2} dt \notag\\
&\qquad = \left( \|\sigma\|_{\HH^s}^2 -2\int_{\TT} u \cdot \nabla \Vort (-\Delta)^{{s}} \Vort dx \right) dt 
+ 2  \sum_{l\in\ZZF} \left( \int_{\TT} \sigma_l (-\Delta)^{{s}} \Vort   dx \right) d W^l,
\label{eq:Ito:Hr}  
\end{align}
where we have used that
\begin{align*}
\frac{d}{dt} |k|^{2{s}(t)} 
= 2\dot{\alpha}(t) \log(|k|) |k|^{2{s}(t)}
=
\begin{cases}
2 m T_m^{-1} \log(|k|) |k|^{2{s}(t)}, & t \in [0, T_m],\\
0, & t > T_m.
\end{cases}
\end{align*}
In order to bootstrap from \eqref{eq:Ito:Hr} to compute higher moments we now make a 
second application of the It\={o} lemma with $\phi(x) = x^{q/2}$.  We obtain
\begin{align}
&d \| \Lambda^{{s}} \Vort\|_{L^2}^q + q \|\Lambda^{{s}+\gamma/2} \Vort\|_{L^2}^2  \| \Lambda^{{s}} \Vort\|_{L^2}^{q-2} dt  - q\dot{\alpha}(t) \|  \Lambda^{{s}} (\log \Lambda)^{1/2} \Vort \|_{L^{2}}^{2} \| \Lambda^{{s}} \Vort\|_{L^2}^{q-2}dt \notag\\
&\qquad =\frac{q}{2}\| \Lambda^{{s}}\Vort\|_{L^2}^{q-2} \left( \|\sigma\|_{\HH^s}^2 -2\int_{\TT} u \cdot \nabla \Vort (-\Delta)^{{s}} \Vort dx \right) dt 
+ q \| \Lambda^{{s}}\Vort\|_{L^2}^{q-2} \sum_{l\in\ZZF} \left( \int_{\TT} \sigma_l (-\Delta)^{{s}} \Vort   dx \right) d W^l
\notag\\
&\qquad \quad + \frac{q(q-2)}{2}  \|\Lambda^{{s}} \Vort\|_{L^2}^{q-4}\sum_{l\in\ZZF} \left( \int_{\TT} \sigma_l (-\Delta)^{{s}} \Vort   dx \right)^{2}dt.
\label{eq:Ito:Hr:q}
\end{align}

Now, since for any $\gamma>0$ there exists $N_* = N_*(\gamma, m, T_m) >0$ such that
\begin{align}
 \dot{\alpha} \log(|k|) \leq  m T_m^{-1} \log (|k|) \leq \frac{1}{2} |k|^\gamma, \quad \mbox{for all } |k|\geq N_* 
\label{eq:N:m:gamma}
\end{align}
and possibly choosing $N_*$ larger so that $N_*^{\gamma/2} >   m T_m^{-1}$ we have
\begin{align}
 \dot \alpha(t) \|  \Lambda^{\tilde{s}(t)} (\log \Lambda)^{1/2} \Vort \|_{L^{2}}^{2} 
&=    \dot \alpha(t) \sum_{k \in \ZZ^2_0} |k|^{2{s}(t)} \log|k| |\hat\Vort_k|^2 \notag\\
& \leq  \frac{1}{2}\| \Lambda^{{s}(t) + \gamma/2} \Vort\|_{L^2}^2 + N_{*}^{2s-2+\gamma} \indFn{t\leq T_m} \| \Vort\|_{H^1}^2
\label{eq:log:Hs}
\end{align}
Thus, with \eqref{eq:log:Hs} and \eqref{eq:Ito:Hr:q} we may thus conclude
\begin{align}
&d \| \Lambda^{{s}} \Vort\|_{L^2}^q + \frac{q}{2} \|\Lambda^{{s}+\gamma/2} \Vort\|_{L^2}^2  \| \Lambda^{{s}} \Vort\|_{L^2}^{q-2} dt  \notag\\
&\qquad \leq \frac{q}{2}
\|\Lambda^{s} \Vort\|_{L^2}^{q-2} \left( (q-1) \|\sigma\|_{\HH^s}^2 
+ N_{*}^{2s-2+\gamma} \indFn{t\leq T_m} \| \Vort\|_{H^1}^2
+ 2 \left|\int_{\TT} u \cdot \nabla \Vort (-\Delta)^{{s}} \Vort dx\right| \right) dt 
\notag\\
&\qquad \quad + q \| \Lambda^{{s}}\Vort\|_{L^2}^{q-2} \langle \Lambda^{{s}} \sigma, \Lambda^{{s}}\Vort \rangle dW.
\label{eq:Ito:Hr:q:2}
\end{align}

To estimate the nonlinear term on the right side of \eqref{eq:Ito:Hr:q}, since $\nabla \cdot u = 0$, we may rewrite
\begin{align*}
\int_{\TT} (K\ast \Vort) \cdot \nabla \Vort (-\Delta)^{{s}} \Vort dx = \int [\Lambda^{{s}}, (K\ast \Vort) \cdot \nabla] \Vort \Lambda^{{s}} \Vort dx
\end{align*}
where $[\Lambda^{{s}}, f \cdot \nabla] g= \Lambda^{{s}} (f \cdot \nabla g) - f \cdot \Lambda^{{s}} g$.
We now use the commutator estimate \eqref{eq:comm:sobolev} of Lemma~\ref{lem:commutator}, with $\eps = \frac{1}{4}$, and conclude
\begin{align}
2 \left|\int_{\TT} (K\ast \Vort) \cdot \nabla \Vort (-\Delta)^{{s}} \Vort dx \right| \leq C (1 + \|\Vort\|_{H^1}^{p}) + \frac{1}{4}\|\Vort\|_{H^{{s}+\gamma/2}}^{2}.
\label{eq:com:Est:Smoothing}
\end{align}
with $C = C(m)$ since ${s} \leq r+m$ and where 
\begin{align}
p =   \frac{ 4 ( (4+\gamma) (r+m +\gamma) - 4 )}{\gamma (6 + \gamma)}.
\label{eq:q:gamma}
\end{align}
Putting together \eqref{eq:com:Est:Smoothing} with \eqref{eq:Ito:Hr:q:2} we obtain
with an appropriate usage of  the $\eps$-Young inequality.
\begin{align}
&d \| \Lambda^{{s}} \Vort\|_{L^2}^q + \frac{q}{8} \|\Lambda^{{s}+\gamma/2} \Vort\|_{L^2}^2  \| \Lambda^{{s}} \Vort\|_{L^2}^{q-2} dt  \notag\\
&\qquad \leq C \left(   N_{*}^{2r + 2m -2+\gamma} \| \Vort\|_{H^1}^2  
+ \|\sigma\|_{\HH^{r+m}}^2 + (1+ \| \Vort\|_{H^1}^{p}) \right)^{q/2} dt 
+ q \| \Lambda^{{s}}\Vort\|_{L^2}^{q-2} \langle \Lambda^{{s}} \sigma, \Lambda^{{s}}\Vort \rangle d W. 
\label{eq:Ito:Hr:3}  
\end{align}
We now conclude the desired results by integrating \eqref{eq:Ito:Hr:3} over any interval $[0,t] \subset [0,T]$, taking a supremum in $t$ and then taking expected values. 
Applying standard arguments using the Burkholder-Davis-Gundy inequality to the martingale terms, we obtain
\begin{align}
\E& \left( \sup_{t \in [0,T]} \|\Lambda^{{s(t)}} \Vort(t) \|_{L^2}^q 
+  \int_0^T \|\Lambda^{{s(t)}+\gamma/2} \Vort\|_{L^2}^2  \| \Lambda^{{s(t)}} \Vort\|_{L^2}^{q-2} dt  \right)
\notag\\
& \quad \leq C\E \|\Lambda^r \Vort_0\|_{L^2}^q + C \E \left( \int_0^T  1+ \|\sigma\|_{\HH^{{r+m}}}^q + 
   \| \Vort\|_{\HH^1}^{\frac{pq}{2}} ds \right)
\notag\\
& \quad \leq C\E\left( \|\Lambda^r \Vort_0\|_{L^2}^q + \| \Vort_{0}\|_{H^1}^{\frac{pq}{2}} 
+ \|\Vort_0\|_{L^{\frac{pp_\gamma q}{4}}}^{\frac{p p_\gamma q}{4}} \right) 
+ C T\left(1+ \|\sigma\|_{\HH^{{r+m}}}^q + \| \sigma\|_{\HH^1}^{\frac{pq}{2}} 
+ \|\sigma\|_{\LL^{\frac{pp_\gamma q}{4}}}^{\frac{p p_\gamma q}{4}} \right).
\label{eq:KB:moment:3}
\end{align}
for any $T> 0$ and $q\geq 2$, with $C=C(q,m,T_m)$, where $p$ is given by \eqref{eq:q:gamma} and $p_\gamma$ is given by \eqref{eq:p:gamma}.

\subsection{Exponential Moments in \texorpdfstring{$L^p$, $p\geq 2$}{Lp, p > 2}}
\label{sec:Vort:exp}
The purpose of this section is to establish exponential moments for the $L^p$ norms of the solution, i.e.  prove Theorem~\ref{thm:exponential}.

\subsubsection{Pointwise in time exponential moments}
\label{sec:Vort:L2} 
In this subsection we obtain pointwise in time exponential moment bounds for the $L^p$ norms of solutions, that grow only linearly in time, i.e. estimate \eqref{eq:expo:mom:I}. 

For $p\geq 2$ and $\kappa > 0$, to be determined, we now consider the function
\begin{align} 
\psi_{\kappa}(x) = \exp \left( \kappa  (1+  x)^{2/p}\right) \label{eq:phi}
\end{align}
which is smooth in a neighborhood of $[0,\infty)$. We note that
\begin{align*} 
&\psi'_\kappa(x) = \frac{2\kappa }{p } (1+  x)^{\frac{2-p}{p}} \psi_\kappa(x)  \notag\\
&\psi''_\kappa(x) = -\frac{2\kappa  (p-2)}{p^2} (1+  x)^{\frac{2-2p}{p}} \psi_ \kappa(x) 
+ \frac{4 \kappa^{2}}{p^2} (1+  x)^{\frac{2(2-p)}{p}}\psi_ \kappa(x) \leq  \kappa  (1+  x)^{\frac{2-p}{p}} \psi_ \kappa'(x).
\end{align*}
Let $x(t) = \| \Vort(t) \|_{L^p}^p$. By the It\=o Lemma, and \eqref{eq:pk:Ito} we have that 
\begin{align} 
d \psi_ \kappa(x) = \psi'_ \kappa(x) \left(- pT_{1,p} + \frac{p(p-1)}{2} T_{2,p} \right)dt 
	+ \frac{p^2}{2} \psi''_ \kappa(x) \sum_{l \in \ZZF}    S_{l,p}^2 dt 
	+ p \sum_{l \in \ZZF} \psi'_ \kappa (x) S_{l,p} dW^l \label{eq:Ito:Lp:2}
\end{align}
Using \eqref{eq:CC:lower:bound} and \eqref{eq:pk:sigma:upper} we find 
\begin{align*}
- p T_{1,p} + \frac{p(p-1)}{2} T_{2,p} \leq - \frac{1}{C_\gamma} \| \Vort\|_{L^{p}}^{p} + C  \|\sigma\|_{\LL^{p}}^{p}. 
\end{align*}
with $C_\gamma\geq 1$, while the H\"older inequality implies
\begin{align*}
 \sum_{l \in \ZZF}    S_{l,p}^2 \leq \sum_{l\in \ZZF} \left( \int_{\TT} \sigma_l \Vort^{p-1} dx \right)^2
 \leq \| \sigma\|_{\LL^p}^2 \| \Vort \|_{L^{p }}^{2(p -1)}
\end{align*}
Hence we obtain
\begin{align}
d \psi_ \kappa(x) 
&\leq  \psi'_ \kappa(x) \left( - \frac{1}{C_\gamma} \| \Vort\|_{L^{p}}^{p} + C \|\sigma\|_{\LL^{p}}^{p} +  \kappa  (1+ \| \Vort \|_{L^{p}}^{p} )^{\frac{2-p}{p}} \| \sigma\|_{\LL^{p}}^2 \| \Vort \|_{L^{p}}^{2(p-1)}  \right)dt \notag\\
&\qquad + p \psi'_ \kappa (x)  \sum_{l \in \ZZF}S_{l,p} dW^l \notag\\
&\leq  \psi'_ \kappa(x) \left( - \frac{1}{C_\gamma} \| \Vort\|_{L^{p}}^{p} + C  \|\sigma\|_{\LL^{p}}^{p} +  \kappa  \| \sigma\|_{\LL^{p}}^2  \| \Vort \|_{L^{p}}^{p}  \right)dt + p \psi'_ \kappa (x)  \sum_{l \in \ZZF}S_{l,p} dW^l
  \label{eq:psi:eps:1}
\end{align}
Now for any $ \kappa =  \kappa(p,  \| \sigma\|_{\LL^{p}})$ sufficiently small so that 
\begin{align}
\kappa   \left( 1+ C_\gamma \| \sigma\|_{\LL^{p}}^2 \right)  \leq \frac 12,
\label{eq:eps:cond:1}
\end{align}
where $C(p)$ is the constant in \eqref{eq:psi:eps:1},
and any $T>0$,  by integrating \eqref{eq:psi:eps:1} we find
\begin{align}
   \E \psi_ \kappa(x(T) ) \leq \E\psi_ \kappa(x(0)) + \E \int_0^T  \psi'_ \kappa (x(t))   
   \left( - \frac{1}{2 C_\gamma} \| \Vort\|_{L^{p}}^{p} + C   \|\sigma\|_{\LL^{p}}^{p} \right) ds.
   \label{eq:Lp:exp:1}
\end{align}
Using \eqref{eq:eps:cond:1} we next estimate
\begin{align}
\psi'_ \kappa (x) \left( - \frac{1}{4 C_\gamma} \| \Vort\|_{L^{p}}^{p} + C    \|\sigma\|_{\LL^{p}}^{p} \right) 
&\leq\frac{2  \kappa }{p }  \exp\left(\kappa  (1 + \| \Vort\|_{L^{p}}^{p})^{2/p} \right) \left( - \frac{1}{4 C_\gamma} \| \Vort\|_{L^{p}}^{p } + C \|\sigma\|_{\LL^{p}}^{p} \right) \notag\\
&\leq \frac{2  \kappa }{p }  \exp\left( \kappa (1 + 4 C C_\gamma    \| \sigma\|_{\LL^{p}}^{p})^{2/p} \right) \leq C  \kappa .
\label{eq:Lp:exp:2}
\end{align}
To see this one has to treat separately the cases when $\|\Vort\|_{L^{p}}^{p}$ is larger or smaller than $4 C C_\gamma \| \sigma \|_{\LL^{p}}^{p}$.
 Combining \eqref{eq:Lp:exp:1} with \eqref{eq:Lp:exp:2} we obtain that 
\begin{align}
\E \exp\left(  \kappa   \| \Vort(T)\|_{L^{p}}^{2} \right) 
+ \frac{1}{4 p C_\gamma} \E \int_0^T  \exp\left(   \kappa   \| \Vort(t)\|_{L^{p}}^{2} \right) dt  \leq C \E \exp\left(  \kappa   \| \Vort_0\|_{L^{p}}^{2} \right)  +  C \kappa T 
\label{eq:Lp:exponential}
\end{align}
where $ \kappa =  \kappa(p,  \|\sigma \|_{\LL^{p}}^{-1})$ is such that \eqref{eq:eps:cond:1} holds, and $C = C (p)$. We note that the right side of \eqref{eq:Lp:exponential} grows only linearly in $p$.

\subsubsection{Exponential moments for the time-integral}
In this subsection we prove estimates for the exponential of the time integral of the $L^p$ norms, i.e. bound \eqref{eq:expo:mom:II}.  For this purpose, let $p\geq2 $ be even and 
\begin{align*} 
X(t) =  (1 +  \| \Vort(t)\|_{L^p}^p)^{2/p} + \eps \int_0^t \| \Vort(s)\|_{L^p}^{2} ds
\end{align*}
where $\eps = \eps(p) > 0$ is to be determined later,
and apply the It\=o lemma to the $C^2$ function 
\[ 
\psi_\kappa(X) = \exp\left( \kappa X \right).
\] 
In order to do this, we first use \eqref{eq:pk:Ito} and obtain
\begin{align} 
d X &=   \frac{2}{p} (1 + \|\Vort\|_{L^p}^p)^{\frac{2-p}{p}} d\|\Vort\|_{L^p}^p - \frac{ (p-2)}{p^2} (1+ \|\Vort\|_{L^p}^p)^\frac{2-2p}{p} d\|\Vort\|_{L^p}^p d \|\Vort\|_{L^p}^p + \eps  \|\Vort\|_{L^p}^2 dt\notag\\
&= \frac{2}{p} (1 + \|\Vort\|_{L^p}^p)^{\frac{2-p}{p}}  \left( - p T_{1,p} dt + \frac{p(p-1)}{2} T_{2,p} dt + p \langle \sigma,  \Vort^{p-1} \rangle dW \right) \notag\\
&\qquad - (p-2) (1+ \|\Vort\|_{L^p}^p)^\frac{2-2p}{p} |\langle \sigma, \Vort^{p-1} \rangle|^2 dt + \eps \|\Vort\|_{L^p}^2 dt.
\label{eq:ITO:X}
\end{align}
Here we used the shorthand notation $\langle \sigma, g \rangle dW = \sum_{l\in\ZZF} \int_{\TT} \sigma_l g dx dW^l$.
Therefore, the It\=o lemma applied to $\psi_\kappa(X)$ yields
\begin{align} 
d \psi_\kappa(X) &= \psi_\kappa'(X) dX + \frac 12 \psi_{\kappa}''(X) dX dX \notag\\
&= \kappa \psi_\kappa(X) \Bigl(   - 2  (1 +  \|\Vort\|_{L^p}^p)^{\frac{2-p}{p}} T_{1,p} +  (p-1) (1 +  \|\Vort\|_{L^p}^p)^{\frac{2-p}{p}}  T_{2,p} \notag\\
&\qquad \qquad \qquad +\eps\|\Vort\|_{L^p}^2   - (p-2) (1+ \|\Vort\|_{L^p}^p)^\frac{2-2p}{p} |\langle \sigma, \Vort^{p-1} \rangle|^2\Bigr)dt \notag\\
&\qquad + \kappa \psi_\kappa(X)  2 (1 + \|\Vort\|_{L^p}^p)^{\frac{2-p}{p}}   \langle \sigma,\Vort^{p-1} \rangle dW  \notag\\
&\qquad +  \kappa^2  \psi_\kappa(X)   (1+  \|\Vort\|_{L^p}^p)^\frac{2-p}{p} \sum_{l\in \ZZF} \left(     \int \sigma_l  \Vort^{p-1} dx\right)^2 dt. 
\label{eq:ITO:X:2}
\end{align}
Using \eqref{eq:CC:lower:bound}--\eqref{eq:pk:sigma:upper}, the H\"older inequality, and the definition of $\psi_\kappa$ we thus infer
\begin{align} 
d \psi_\kappa(X) &\leq \kappa \psi_\kappa(X) \Bigl(   - \frac{1}{C_\gamma} (1 + \|\Vort\|_{L^p}^p)^{\frac{2-p}{p}} \|\Vort\|_{L^p}^p +C \|\sigma\|_{\LL^p}^2+ \eps \|\Vort\|_{L^p}^2  +    C \kappa   \|\sigma\|_{\LL^p}^2 \|\Vort\|_{L^p}^2 \Bigr) dt \notag\\
&\qquad + 2 \kappa \psi_\kappa(X)  (1 + \|\Vort\|_{L^p}^p)^{\frac{2-p}{p}}   \langle \sigma,\Vort^{p-1} \rangle dW .
\label{eq:ITO:X:3}
\end{align}
Next, we estimate
\begin{align*} 
- \frac{1}{C_\gamma} (1 + \|\Vort\|_{L^p}^p)^{\frac{2-p}{p}} \|\Vort\|_{L^p}^p 
&= - \frac{1}{C_\gamma} (1 + \|\Vort\|_{L^p}^p)^{\frac{2-p}{p}} (1 + \|\Vort\|_{L^p}^p) +  \frac{1}{C_\gamma} (1 + \|\Vort\|_{L^p}^p)^{\frac{2-p}{p}}  \notag\\
&= - \frac{1}{C_\gamma} (1 + \|\Vort\|_{L^p}^p)^{2/p} + 1 \leq - \frac{1}{C_\gamma} \|\Vort\|_{L^p}^2 + 1
\end{align*}
since  $C_\gamma \geq 1$ and $p\geq 2$. We thus obtain
\begin{align} 
d \psi_\kappa(X) 
&\leq \kappa \psi_\kappa(X) \Bigl( -  \frac{1}{C_\gamma} \|\Vort\|_{L^p}^2  + C( 1+ \|\sigma\|_{\LL^p}^2) +  \eps \|\Vort\|_{L^p}^2  +  C \kappa   \|\sigma\|_{\LL^p}^2 \|\Vort\|_{L^p}^2 \Bigr) dt \notag\\
&\qquad + 2 \kappa \psi_\kappa(X)  (1 + \|\Vort\|_{L^p}^p)^{\frac{2-p}{p}}   \langle \sigma,\Vort^{p-1} \rangle dW.
\label{eq:ITO:X:4}
\end{align}
We next choose $\kappa, \eps\in (0,1]$  to be sufficiently small so that
\begin{align} 
&C  \kappa  (1 + \|\sigma\|_{\LL^p}^2) \leq \frac{1}{2C_\gamma}, \qquad  \eps \leq \frac{1}{2C_\gamma},
\label{eq:eps:kappa}
\end{align}
where $C_\gamma$ is the constant appearing next to the negative term on the right side of \eqref{eq:ITO:X:4}.
With this choice of $\eps$ and $\kappa$, we may now integrate \eqref{eq:ITO:X:4} on $[0,t]$, take expected values  to obtain
\begin{align} 
\E \psi_\kappa(X(t)) \leq \E \psi_\kappa(X(0))   +  \int_0^t \psi_\kappa(X(s)) ds.
\label{eq:Lp:exp:4}
\end{align}
The constant ($1$) in front of the second term on the right side of \eqref{eq:Lp:exp:4} is independent of $\sigma$ because due to \eqref{eq:eps:kappa} we have $C \kappa ( 1+ \|\sigma\|_{\LL^p}^2) \leq 1$.
From \eqref{eq:Lp:exp:4} and the Gr\"onwall inequality we infer
\begin{align} 
\E \exp \left(  \kappa \| \Vort(t)\|_{L^p}^2 +  \eps \kappa  \int_0^t \| \Vort(s)\|_{L^p}^{2} ds \right) \leq \E \psi_\kappa(X(T)) \leq e^T \E \psi_\kappa(X(0))
\label{eq:Lp:exp:mom:full}
\end{align}
for any $T>0$, and in particular
\begin{align} 
\E \exp \left(  \eps \kappa   \int_0^T \|\Vort(s)\|_{L^p}^2 ds   \right) 
\leq  e^T \E \exp\left(   \kappa (1 +\|\Vort_0\|_{L^p}^p)^{2/p} \right)
\label{eq:Lp:exp:mom}
\end{align}
holds for $T>0$ and $\eps,\kappa$ chosen such that \eqref{eq:eps:kappa} holds. Note that without loss of generality we may take $\eps = \kappa$, which proves \eqref{eq:expo:mom:II}.  

\section{The Feller Property. Existence and Regularity Properties of Invariant Measures}
\label{sec:Feller:existence}
\setcounter{equation}{0}

In this section we apply the a priori moment bounds derived in the Section~\ref{sec:moments}
to establish that $\{P_t\}_{t \geq 0}$ is Feller and then to infer the existence and regularity
properties of invariant measures for the dual semigroup.

\subsection{Instantaneous Smoothing and the Feller Property}
\label{sec:feller}

In order to establish the Feller property for $P_t$ we will need to following continuous
dependence estimates which rely on a smoothing properties of \eqref{eq:frac:SNSE} 
established in Theorem~\ref{thm:smoothing}
\begin{proposition}[\bf Continuous Dependence in $H^r$]
\label{prop:cont:dep}
Fix $m \geq r$, $T_m > 0$ and define
\begin{align} 
\alpha(t)  = 
\begin{cases}
tm T^{-1}_m, & t \in [0, T],\\
m, & t > T_m.
\end{cases}
\label{eq:alpha:growth:for:T}
\end{align}
Then for any $\Vort_0, \bar\Vort_0 \in H^r$ and for any   $\tau \geq 0$
 \begin{align} 
    \| \Vort(\tau, \Vort_0) - &\Vort(\tau, \bar\Vort_0) \|_{H^{r-1+\alpha(\tau)}}^2 
    \notag\\
     &\leq \| \Vort_0 - \bar \Vort_0\|_{H^{r-1}}^2 \exp\left( C \int_0^\tau 1 + \|\Vort(t,\Vort_0)\|_{H^{r+\alpha(t)}}^2 + \|\Vort(t, \bar{\Vort}_0)\|_{H^{r+\alpha(t)}}^2 dt \right)
     \label{eq:cont:est}
\end{align}
a.s., where the deterministic constant $C$ depends on $r, m, \gamma, T_m$ but is independent 
of $\tau$ and  $\Vort_0, \bar\Vort_0 \in H^r$.
\end{proposition}
\begin{proof}[proof of Proposition~\ref{prop:cont:dep}]
For brevity of notation let $\Vort(t) = \Vort(t, \Vort_0)$, $\bar{\Vort}(t) = \Vort(t, \bar{\Vort}_0)$ and $\rho = \Vort - \bar{\Vort}$.   
Then $\rho$ satisfies:
\begin{align}
  \pd{t} \rho + \Lambda^\gamma \rho + B(\rho, \Vort) + B(\bar\Vort, \rho) =0, \quad \rho(0) =  \Vort_0 - \bar \Vort_0.
\label{eq:cont:eqn}
\end{align}
As in the proof of Theorem~\ref{thm:smoothing} we denote 
\begin{align*}
  s(t) = r - 1 + \alpha(t)
\end{align*}
and find that
\begin{align*}
  \frac{1}{2}\ddt \| \Lambda^s \rho\|^{2} + \|\Lambda^{s + \gamma/2} \rho\|^2
  -& 2 \dot{\alpha}(t) \|  \Lambda^{{s}} (\log \Lambda)^{1/2} \rho \|^2
  \notag\\
   &= 
      - \langle \Lambda^s B(\rho, \Vort), \Lambda^s \rho \rangle - \langle \Lambda^s  B(\bar\Vort, \rho), \Lambda^s \rho \rangle
\end{align*}
Repeating the computations leading to \eqref{eq:log:Hs} we infer
\begin{align}
 \ddt \| \Lambda^s \rho\|^{2} + \|\Lambda^{s + \gamma/2} \rho\|^2
 \leq& C \|\Lambda^s \rho\|^2 + 
      | \langle \Lambda^s B(\rho, \Vort), \Lambda^s \rho \rangle | +  | \langle \Lambda^s  B(\bar\Vort, \rho), \Lambda^s \rho \rangle |
      \notag\\
      =&  C \|\Lambda^s \rho\|^2 + T_1 + T_2.
 \label{eq:lambda:Hs:eqn}
\end{align}
for a constant $C = C(m, T_m)$.

For the first term $T_1$, we use \eqref{eq:Hs:product} and infer
\begin{align}
|T_1| 
&\leq \|\Lambda^s ((K \ast \rho) \cdot \nabla \Vort)\|_{L^2}\|\Lambda^s \rho\|_{L^2} \notag\\
&\leq C\left(\| \Lambda^s (K \ast \rho) \|_{L^2} \| \nabla \Vort\|_{L^\infty}+ \| \Lambda^s \nabla \Vort\|_{L^2} \|  K \ast \rho \|_{L^\infty} \right) \|\Lambda^s \rho\|_{L^2} \notag\\
&\leq C\left(\| \Lambda^{s-1} \rho \|_{L^2} \| \Lambda^{s+1} \Vort\|_{L^2}  + \| \Lambda^{s+1} \Vort\|_{L^2} \|   \Lambda^{s-1} \rho \|_{L^2} \right) \|\Lambda^s \rho\|_{L^2} \notag\\
&\leq C\| \Lambda^{r+\alpha} \Vort\|_{L^2}  \|\Lambda^s \rho\|_{L^2}^2.
\label{eq:T1:est:cont:dep}
\end{align}
On the other hand, for $T_2$ we take advantage cancelations and make use of the commutator estimate \eqref{eq:comm:abstract} to bound
\begin{align}
|T_2| 
&\leq C \|\Lambda^s ((K \ast \bar \Vort) \cdot \nabla \rho) - (K \ast \bar\Vort ) \cdot \nabla (\Lambda^s \rho)\|_{L^2}\|\Lambda^s \rho\|_{L^2} \notag\\
&\leq C \left( \| \nabla (K \ast \bar \Vort) \|_{L^{8/\gamma}}   \|  \Lambda^{s} \rho  \|_{L^{8/(4-\gamma)}} +   \| \Lambda^s (K \ast \bar \Vort)   \|_{L^{8/\gamma}}  \| \nabla \rho \|_{L^{8/(4-\gamma)}}   \right) \|\Lambda^s \rho\|_{L^2}  \notag\\
&\leq C \left( \| \ \bar \Vort \|_{L^{8/\gamma}}   \|  \Lambda^{s+\gamma/4} \rho \|_{L^{2}} +   \| \Lambda^{s-1}  \bar \Vort   \|_{L^{8/\gamma}}  \|  \Lambda^{\gamma/4} \nabla \rho\|_{L^{2}}   \right) \|\Lambda^s \rho\|_{L^2}  \notag\\
&\leq C \| \Lambda^{s+1} \bar{\Vort} \|_{L^2} \|\Lambda^{s + \gamma/2} \rho\|_{L^2}  \|\Lambda^s \rho\|_{L^2} \notag\\
&\leq \frac{1}{2}  \|\Lambda^{s + \gamma/2} \rho\|_{L^2}^2 + C\| \Lambda^{r + \alpha} \bar{\Vort}\|_{L^2}^2 \|\Lambda^s \rho\|_{L^2}^2.
\label{eq:T2:est:cont:dep}
\end{align}
Combining \eqref{eq:T1:est:cont:dep} and \eqref{eq:T2:est:cont:dep} with \eqref{eq:lambda:Hs:eqn} and applying the Gr\"onwall lemma
yields the desired result \eqref{eq:cont:est}.
\end{proof}

With Proposition~\ref{prop:cont:dep} now, in hand we now turn to establish the Feller property.  Note that, since growth of the distance between two solutions 
is controlled by the growth of \emph{each} of the individual solutions in \eqref{eq:cont:est}, we need to make more careful use of stopping time arguments to
establish Feller property than for the 2D Navier-Stokes equations on the $L^2$ phase space.  

\begin{proposition}
\label{prop:Feller}
The Markov semigroup is Feller on $H^r$ for any $r > 2$, i.e.
\begin{align*}
  P_T : C_b(H^r) \to C_b(H^r)
\end{align*}
for any $T \geq 0$.
\end{proposition}
\begin{proof}[Proof of Proposition~\ref{prop:Feller}]
Let $\phi \in C_b$ and $T> 0$ be given.  Fix $\Vort_0 \in H^r$.
For any $\bar{\Vort}_0$ in $H^r$ and $K > 0$ we define the stopping times
\begin{align*}
  &\tau_K(\bar{\Vort}_0) := \inf_{t \geq 0} \left\{ \sup_{s \in [0,t]} \|\Vort(s, \bar{\Vort}_0)\|^2_{H^{r +\alpha(s)}} 
     + \int_0^t \|\Vort(t, \bar{\Vort}_0)\|^2_{H^{r + \gamma/2 +\alpha(t)}} dt  \geq \kappa_1 t + K\right\},
\end{align*}
where $\alpha$ is defined precisely as in \eqref{eq:alpha:growth:for:T} with $m = 1$ and $T_m = T$.  Here $\kappa_1 = C(T) \poly(\|\sigma\|_{\HH^{r+m}})$ is the constant appearing in \eqref{eq:Smoothing:moments}
corresponding to $T_m = T$ and $m = 1$.  In particular we emphasize that $\kappa_1$ is independent of $\bar{\Vort}_0$. We let \[ \tau_K(\Vort_0,\bar{\Vort}_0) = \tau_K(\Vort_0) \wedge \tau_K(\bar{\Vort}_0).\] 
Observe that, for any fixed $K > 0$, it immediately follow from Proposition~\ref{prop:cont:dep}, \eqref{eq:cont:est} that
 \begin{align} 
    \| \Vort(\tau_K(\Vort_0, \bar{\Vort}_0) \wedge T, \Vort_0) &- \Vort(\tau_K(\Vort_0, \bar{\Vort}_0) \wedge T, \bar\Vort_0) \|_{H^{r-1+\alpha(\tau_K(\Vort_0, \bar{\Vort}_0) \wedge T)}}^2 
    \notag\\
     &\leq \| \Vort_0 - \bar \Vort_0\|_{H^{r-1}}^2 \exp \left(C(T) \left( T(1+2\kappa_1) + 2K\right)\right).
     \label{eq:cont:est:tau_K}
\end{align}
On the other hand, making use of the estimate \eqref{eq:Smoothing:moments} we have
\begin{align}
  \Prb( \tau_K(\Vort_0, \bar{\Vort}_0) < T) 
  \leq&  \Prb( \tau_K(\Vort_0) < T) +  \Prb( \tau_K(\bar{\Vort}_0) < T)
  \notag\\
  \leq&   \Prb\left( \sup_{s \in [0,T]} \|\Vort(s, \Vort_0)\|^2_{H^{r +\alpha(s)}}  + \int_0^T \|\Vort(t, \Vort_0)\|_{H^{r+\alpha(t)}}^2 dt \geq \kappa_1T+ K \right) 
  \notag\\
           &+ \Prb \left( \sup_{s \in [0,T]} \|\Vort(s, \bar{\Vort}_0)\|^2_{H^{r +\alpha(s)}}  +\int_0^T \|\Vort(t,\bar\Vort_0)\|_{H^{r+\alpha(t)}}^2 dt \geq \kappa_1T+ K \right)
           \notag\\
  \leq& \frac{C(T) \poly(\|\Vort_0\| + \|\bar{\Vort}_0 \|)}{K}
  \label{eq:growth:est:tau_K}
\end{align}
where we note that $C$ is in particular independent of $K > 0$.
 Finally, for any $\bar\Vort_0 \in H^r$,
 \begin{align}
  |P_T \phi(\Vort_0) - P_T \phi(\bar\Vort_0)|
  =&   |\E \phi(\Vort(T, \Vort_0)) - \E \phi(\Vort(T,\bar\Vort_0))|
  \notag\\
  \leq&   |\E \phi(\Vort(T, \Vort_0)) - \E \phi(\Vort(T,\bar\Vort_0)) \indFn{\tau_K(\Vort_0, \bar{\Vort}_0) \geq T}|
  + 2\|\phi\|_\infty \Prb( \tau_K(\Vort_0, \bar{\Vort}_0) < T) 
  \notag\\
    \leq&   |\E \left(\phi(\Vort(T \wedge \tau_K(\Vort_0, \bar{\Vort}_0), \Vort_0)) - \E \phi(\Vort(T \wedge \tau_K(\Vort_0, \bar{\Vort}_0),\bar\Vort_0))\right)
    \indFn{\tau_K(\Vort_0, \bar{\Vort}_0) \geq T}|
    \notag\\
  &+ 2\|\phi\|_\infty \Prb( \tau_K(\Vort_0, \bar{\Vort}_0) < T).
  \label{eq:Feller:est:1}
\end{align}

Using estimates \eqref{eq:cont:est:tau_K}--\eqref{eq:Feller:est:1}, we now establish the desired continuity as follows.
Let $\eps >0$ be given. In view of \eqref{eq:growth:est:tau_K} we may choose $K$ such that 
\begin{align}
  2 \|\phi\|_\infty \Prb( \tau_K(\Vort_0, \bar{\Vort}_0) < T) \leq \epsilon/4
   \label{eq:eps:stopping:time}
\end{align}
for any $\bar{\Vort}_0 \in B_{H^r}(\Vort_0,1)$.
Having fixed $K$, we next use the Relich and Stone theorems and pick $\tilde{\phi} \in C_b(H^r)$, Lipschitz continuous, such that
\begin{align}
  \sup_{\bar{\Vort} \in B_{H^{r+1}}(\kappa_1T+K, 0) } |\tilde{\phi}(\bar{\Vort}) - \phi(\bar{\Vort})| \leq \eps/4.
     \label{eq:eps:lip:approx}
\end{align}
With these choices we apply the observations in \eqref{eq:eps:stopping:time}, \eqref{eq:eps:lip:approx} to \eqref{eq:Feller:est:1}
and using \eqref{eq:cont:est:tau_K}, we obtain, 
\begin{align*}
  |P_T &\phi(\Vort_0) - P_T \phi(\bar\Vort_0)| 
\notag\\
  \leq& 
    \frac{3\eps}{4} 
  + \|\nabla \tilde\phi\|_\infty
   \E \left(\| \Vort(T \wedge \tau_K(\Vort_0, \bar{\Vort}_0), \Vort_0) -\Vort(T \wedge \tau_K(\Vort_0, \bar{\Vort}_0),\bar\Vort_0)\|_{H^r}
    \indFn{\tau_K(\Vort_0, \bar{\Vort}_0) \geq T} \right)
    \notag\\
  \leq& \frac{3\eps}{4} 
  + \|\nabla \tilde\phi\|_\infty
 \| \Vort_0 - \bar \Vort_0\|_{H^{r}}^2 \exp \left(C(T)\left( T(1+2\kappa_1) + 2K\right)\right).
\end{align*}
for any $\bar{\Vort}_0 \in B_{H^r}(\Vort_0,1)$ and where we denote the Lipschitz constant 
associated with $\tilde \phi$ by 
$\|\nabla \tilde \phi\|_\infty$.  Thus, by choosing 
\begin{align*}
  \delta = \|\nabla \tilde\phi\|_\infty^{-1} \exp \left(-C(T)\left( T(1+2\kappa_1) + 2K\right)\right) \wedge 1
\end{align*}
we infer that $|P_T \phi(\Vort_0) - P_T \phi(\bar\Vort_0)| < \eps$ whenever
$\bar{\Vort}_0 \in B_{H^r}(\delta, \Vort_0)$.    Since $\eps>0$ was arbitrary,
the proof of Proposition~\ref{prop:Feller} is now complete. 
\end{proof}

Proposition~\ref{prop:Feller} is now used to establish the existence of an ergodic invariant measure
with classical arguments.
\begin{proof}[proof of Theorem~\ref{thm:existence}]
The existence of an invariant measure follows from Theorem~\ref{thm:smoothing} by showing that
the sequence of time average measures
\begin{align*}
  \mu_{T}(\cdot) := \frac{1}{T} \int_{0}^{T} P_{t}(0, \cdot) dt = \frac{1}{T} \int_{0}^{T} \Prb( \Vort(t, 0) \in \cdot) dt
\end{align*}
is tight in $Pr(H^{r})$.   For this point the linear time growth bound in \eqref{eq:Smoothing:moments} is crucial. The weak sub-sequential limit is then easily seen to be invariant with the aid of the Feller property.

Having shown that the set of invariant measures $\II$ is non-empty the existence of an \emph{ergodic} measure now follows from the following general argument. It is clear from linearity that $\II$ is convex, and due to the Feller property $\II$ is seen to be closed. 
Due to the Proposition~\ref{prop:higher} we can see that the set of invariant measures is tight, and hence $\II$ is compact. By Krein-Millman, $\II$ has an extremal point. Recalling that an invariant measure is ergodic if and only if it is an extremal point of $\II$  
(cf.~\cite{ZabczykDaPrato1996}), we hence conclude the proof of existence of an ergodic invariant measure.

The proof of Theorem~\ref{thm:existence} is therefore complete once the higher regularity properties of $\mu$ 
established.  This is carried out immediately below in Proposition~\ref{prop:higher}.
\end{proof}

\begin{remark}
The above strategy works for multiplicative noise as well.  Indeed, the strategy of proof in
Proposition~\ref{prop:Feller} has also been used recently in \cite{GHKVZ13} to establish the Feller property
for the 3D stochastic primitive equations.
\end{remark}

\subsection{Higher Regularity of Invariant Measures}

We next show, again with the aide of Theorem~\ref{thm:smoothing}
that any invariant measure $\mu$ of \eqref{eq:frac:SNSE} must be supported
on $C^\infty$.  
\begin{proposition}[\bf Higher Regularity] \label{prop:higher}
Fix $r > 0$ and consider an invariant measure $\mu$
for $\{P_t\}_{t \geq 0}$ defined as a semigroup on $H^r$ then, for
any $q \geq 2$,
\begin{align}
   \int_{H^r} \| \Vort\|^q_{H^s} d \mu(\Vort)  < \infty, \quad \mbox{ for every } s \geq r.
   \label{eq:Hs:Reg:Int}
\end{align}
In particular $\mu$ is supported on $H^s$ for every $s \geq r$.
\end{proposition}

It follows from the proof of this proposition that if we let $\Vort_{0,S}$ such that $\Vort_S(t,\Vort_{0,S})$ be a stationary $H^r$ solution of \eqref{eq:frac:SNSE}
then
\begin{align}
\E \|\Vort_S\|_{H^{s}}^q \leq C \PP (\|\sigma\|_{\HH^s}) 
 \label{eq:KB:moment:3:stat}.
\end{align}
for any $q$ and any $s > r$ holds.

\begin{proof}[Proof of Proposition~\ref{prop:higher}]
For any $R > 0$ and any integer $N$ we define
\begin{align*}
  \phi_{R,N}(\Vort) = \| P_N \Vort\|_{H^s}^q \wedge R
\end{align*}
where $P_N$ is the project operator onto $H_N$.  Clearly
$\phi_{R,N} \in C_b(H^r)$ so that by invariance
\begin{align*}
  \int_{H^r} \phi_{R,N}(\Vort) d\mu(\Vort) =  \frac{1}{T-1}\int_1^{T}\int_{H^r} P_t \phi_{R,N}(\Vort) d\mu(\Vort)dt.
\end{align*}
for any $T > 1$. 

Applying Theorem~\ref{thm:smoothing} with $m$ the $s$ given here and 
with $T_m = 1$ we infer, for any $\rho > 0$
\begin{align}
\frac{1}{T-1}\int_1^{T}\int_{B_{H^r}(\rho)} P_t \phi_{R,N}(\Vort_0) d \mu(\Vort_0)dt 
\leq& \frac{1}{T-1}\int_1^{T}\int_{B_{H^r}(\rho)} \E \|\Vort(t, \Vort_0)\|_{H^s}^q d \mu(\Vort_0)dt  
\notag\\
\leq& \frac{1}{T-1}\int_{B_{H^r}(\rho)} \int_0^{T} \E \|\Vort(t, \Vort_0)\|_{H^{r + \alpha(t)}}^q dt d \mu(\Vort_0) 
\notag\\
\leq& \frac{C \PP(\rho) + C T \PP(\|\sigma\|_{\HH^{s}})}{T-1}. 
\label{eq:smoothing:App:good:set}
\end{align}
Here $\alpha$ is defined as in \eqref{eq:alpha:def}.  
On the other hand we have that 
\begin{align}
\frac{1}{T-1}\int_1^{T}\int_{B_{H^r}(\rho)^c} P_t \phi_{R,N}(\Vort) d\mu(\Vort)dt \leq R \mu(B_{H^r}(\rho)^c). 
\label{eq:brute:est:bad:set}
\end{align}
Combining \eqref{eq:smoothing:App:good:set} and \eqref{eq:brute:est:bad:set} we infer that for any
$R, \rho > 0$
\begin{align*}
  \int_{H^r} \phi_{R,N}(\Vort) d\mu(\Vort)  \leq \frac{C \PP(\rho) + C T \PP(\|\sigma\|_{\HH^{s}})}{T-1} + R \mu(B_{H^r}(\rho)^c)
\end{align*}
Since $T > 1$ was arbitrary to begin with 
\begin{align*}
  \int_{H^r} \phi_{R,N}(\Vort) d\mu(\Vort)  \leq  C \PP(\|\sigma\|_{\HH^{s}})+ R \mu(B_{H^r}(\rho)^c),
\end{align*}
so that finally, taking $\rho \to \infty$ we conclude that
\begin{align*}
   \int_{H^r} \phi_{R,N}(\Vort) d\mu(\Vort)  \leq  C \PP(\|\sigma\|_{\HH^{s}}).
\end{align*}
The desired result,  \eqref{eq:Hs:Reg:Int}, now follows from the monotone convergence theorem.
\end{proof}

\section{Gradient Estimates On the Markovian Semigroup}
\label{sec:ASF}
\setcounter{equation}{0}

In this section we carry out certain estimates on the gradient of the Markov semigroup. These estimates are used to establish certain time-asymptotic smoothing properties of these Markov operators, introduced in \cite{HairerMattingly06}. To set up the gradient estimates, we recall the following topological notions from
\cite{HairerMattingly06}.  

Let us define pseudo-metrics
on $H^r$ according to
\begin{align}
  d_\eps(\Vort^1, \Vort^2) := 1 \wedge  \eps^{-1}   \| \Vort^1 - \Vort^2 \|_{H^r},
  \label{eq:distDef}
\end{align}
and take
\begin{align*}
  \mbox{Lip}_\eps := \{ \phi \in C_b(H^r):  \| \phi \|_{d_\eps}  <  \infty\},
\end{align*}
where
\begin{align*}
  \| \phi \|_{d_\eps} 
  := \sup_{\Vort^1 \not = \Vort^2  \in H^r} \frac{|\phi(\Vort^1) - \phi(\Vort^2)|}{d_\eps(\Vort_1, \Vort_2)}.
\end{align*}
Then, for a finite signed Borelian measure $\mu$, we define:
\begin{align*}
  \|\mu\|_{d_\eps} := \sup_{ \|\phi\|_{d_\eps} = 1, \phi \in {\rm Lip}_\eps} \int_{H^r} \phi(\Vort) \mu(d \Vort).
\end{align*}
This is known in some of the literature as the Kantorovich distance associated to $d_{\eps}$. Recall cf.~\cite[Definition 3.8, Remark 3.9]{HairerMattingly06}
\begin{definition}[\bf Asymptotic Strong Feller]\label{def:ASF}
We say that $\{P_t\}_{t \geq 0}$ is asymptotically strong Feller (ASF) at $\tilde{\Vort}_0 \in H^r$ if
\begin{align}
	\lim_{\eta \to 0} \limsup_{n \to \infty}  \sup_{\Vort_0 \in B_\eta (\tilde{\Vort}_0)}  
	\|   P_{t_n}(\Vort_0, \cdot) - P_{t_n}(\tilde{\Vort}_0, \cdot)  \|_{d_{\eps_n}} = 0,
	\label{eq:ASF:top:def}
\end{align}
for some increasing sequence $t_n$ and some $\eps_n \to 0$.
\end{definition}

The goal of this section is to prove that the Markov semigroup $P_t$ associated to \eqref{eq:frac:SNSE} is ASF. Instead of working directly with Definition~\ref{def:ASF} above, it was shown in \cite[Proposition 3.12]{HairerMattingly06} that a sufficient condition for establishing the ASF property of $P_t$, are suitable gradient estimates for $P_t$. These gradient estimates are established in the next proposition, which is the main result of this section.

\begin{proposition}[\bf Asymptotic Strong Feller for sufficiently many forced modes] 
\label{thm:ASF:easy}
Let $r>2$, $\gamma>0$. Then there exists $N = N(\|\sigma\|_{\LL^{6/\gamma}}, \gamma, r)$ such that if 
the ball of radius $N$ in ${\mathbb Z}_0^2$ is fully contained in $\ZZF$, then
\begin{align}
\| \nabla P_t \phi (\Vort_0)\|_{{\mathcal L}(H^r)} \leq C  \frac{\poly(\|\Vort_0\|_{H^r}, \|\sigma\|_{\HH^{r+2}})}{\min_{k \in \ZZF} |q_k|^2 } \expo(\kappa_\gamma^{1/2}, \|\Vort_0\|_{L^{6/\gamma}}) \left( \| \phi\|_{\infty} + \delta(t) \| \nabla \phi \|_{\infty} \right),\label{eq:ASF:easy}
\end{align}
for all $t\geq \gamma^{-1}$,  any $\phi \in C_b^1(H^r)$. Here 
\begin{align*}
\delta(t) = \exp\left( - \frac{t \lambda_N^{\gamma/2}}{16} \right) \to 0 \mbox{ as }t \to \infty,
\end{align*}
and $\kappa_\gamma = ( C \PP(\|\sigma\|_{L^{6/\gamma}}))^{-1}$ is chosen to obey \eqref{eq:kappa:gamma:def} below.
\end{proposition}
Recall that in \eqref{eq:ASF:easy} we have
\begin{align*}
\| \Psi (\Vort_0) \|_{{\mathcal L}(H^r)} = \sup_{\| \xi \|_{H^r}=1} | \nabla \Psi(\Vort_0) \cdot \xi |
\end{align*}
for any $\Psi \in C_b^1(H^r)$, and where $\nabla \Psi (\Vort_0) \cdot \xi$ represents its Frechet derivative at $\Vort_0$ in the direction $\xi$. 

Before turning to the proof of Proposition~\ref{thm:ASF:easy}, we show its connection to establishing the uniqueness of the invariant measure.

\begin{proof}[Proof of Theorem~\ref{thm:uniqueness}]
By Proposition~3.12 in \cite{HairerMattingly06}, the gradient estimate obtained in Proposition~\ref{thm:ASF:easy} implies the asymptotic strong Feller property (cf.~Definition~\ref{def:ASF} above). On the other hand, the weak irreducibly property established in Proposition~\ref{prop:WI} below, shows that $\Vort = 0$ lies in the support of every invariant measure.  Recalling that the collection of invariant measures for $P_t$ is closed, convex, and compact, all of whose extremal elements are ergodic, the uniqueness now follows from Theorem~3.16 in~\cite{HairerMattingly06}.
\end{proof}

The remainder of this section is devoted to the proof of Proposition~\ref{thm:ASF:easy}. We begin by relating \eqref{eq:ASF:easy} to a certain control problem and recall some needed aspects of Malliavin calculus.

\subsection{Some aspects of Malliavin Calculus and the Derivation of the Control Problem}

Observe that for any $\Vort_{0}, \xi \in H^{r}$, $t > 0$ and any $\phi \in C_b^1(H^r)$
we have
\begin{align}
 \nabla P_{t}\phi(\Vort_0)\cdot \xi = \E \nabla \phi(\Vort(t, \Vort_{0})) \cdot \LinI_{0,t} \xi
 \label{eq:MSGInitD}
\end{align}
where $\LinI_{s,t} \xi$ solves
\begin{align}
& \partial_t \rho + \Lambda^\gamma \rho + \nabla B (\Vort(t,\Vort_0) ) \rho = 0,
\notag\\
& \rho(s) = \xi, \label{eq:ASF:pertIc}
\end{align}
and we denote
\begin{align*}
\nabla B (\Vort) \rho = (K\ast \rho) \cdot \nabla \Vort + (K \ast \Vort) \cdot \nabla \rho =  B(\rho, \Vort) + B(\Vort,  \rho)
\end{align*}
with $K = \nabla^\perp (-\Delta)^{-1}$ being the Biot-Savart kernel.

Let us now very briefly recall some elements of Malliavin calculus 
in our setting. See \cite{Nualart2009} (and also, \cite{Nualart2006,Malliavin1997}) for further details.  
One of the central objects of the theory is
the Malliavin derivative
$\DM: L^2(\Omega) \rightarrow L^2(\Omega \times L^2(0,T; L_{2}))$
which acts according to 
\begin{align*}
   \DM F = \sum_{k =1}^{N} \pd{x} f \left(\int_{0}^{T}g_{1}(s) dW, \ldots, \int_{0}^{T}g_{n}(s)dW \right) g_{k}
   \label{eq:SimpleFunctions}
\end{align*}
for ``simple functions'' $\mathcal{S}$ of the form
\begin{align*}
   F = f\left(\int_{0}^{T}g_{1}(s) dW, \ldots, \int_{0}^{T}g_{n}(s)dW \right)
\end{align*}
where $f: \RR^{N} \rightarrow \RR$ is any Schwartz class function, and $g_{1} \ldots g_{n}$ are deterministic
elements in $L^{2}(0,T;L_{2})$.   
 
Similarly, we may extend $\DM$ to operate on vector valued
random variables.  In particular we have
$\DM: L^2(\Omega; H^{r}) \rightarrow L^2(\Omega \times L^2([0,T]; L_{2}); H^{r}) =
L^2(\Omega ; L^2([0,T]; L_{2}( H^{r}))$ acting on simple functions
$\mathcal{S}(H^{r})$ of the form
\begin{align*}
   F = \sum_{k=1}^{M} F_{k} \Vort_{k}; \quad F_{k} \in \mathcal{S}, \Vort_{k} \in H^{r}.
\end{align*}
We may close this operator $\DM$ in the space of such simple function $\mathcal{S}(H^{r})$ under the norm
\begin{align*}
   \| F \|_{1,2}  = \E \| F\|_{H^{r}}^{2} +  \E \| \DM F \|^{2}_{H^{r} \times L^2(0,T; L_{2}) }
                        =\E \| F\|_{H^{r}}^{2} +  \E \| \DM F \|^{2}_{L^2(0,T; L_{2}(H^{r})) }
\end{align*}
and define the space Malliavin-Sobolev space $\DDM^{1,2}(H^{r})$.    

Two central ingredients in the Malliavin calculus are the chain rule
\begin{align}
   \DM \phi(F) = \nabla \phi(F) \cdot \DM F \quad \mbox{ for any } F \in \DDM^{1,2}(H^{r}), \phi \in C^{1}_{b}(H^{r}).
\end{align}
and the Malliavin integration by parts formula
\begin{align*}
   \E \langle \DM F , \Kon \rangle_{L^{2}(0,T;{L_{2}})} =  \E \left( F \int_{0}^{T}\Kon dW \right)
\end{align*}
which holds for any $\Kon \in  L^{2}(\Omega; L^{2}_{loc}([0, T); L_{2})$ and any $F \in \DDM^{1,2}(H^{r})$.  Here
$\int_{0}^{T}\Kon dW$ is the Skorohod integral which in fact is \emph{defined}
by this duality relation.  It coincides with the more classical It\=o
integral when $\Kon$ is adapted to $\{\mathcal{F}_{t}\}_{t \geq 0}$.  See \cite{Nualart2009} 
for further details.

We now define the operator
\begin{align*}
  \LinN_{0,t} \Kon = \lim_{\eps \to 0} \frac{\Vort(t, \Vort_0, \sigma( W + \eps \Kon)) -\Vort(t, \Vort_0, \sigma W) }{\eps}.
\end{align*}
for any $\Kon \in  L^{2}(\Omega; L^{2}_{loc}([0, \infty); L_{2}))$. 
On the other hand (cf. \cite{Nualart2006}) we have
\begin{align*}
  \langle \DM \Vort, \Kon \rangle_{L^{2}(0,t;{L_{2}})} = \LinN_{0,t} \Kon.
\end{align*}
Thus, according to the Malliavin chain rule and integration by parts formula
\begin{align}
  \E ( \nabla\phi( \Vort(t, \Vort_0)) \cdot \LinN_{0,t} \Kon)
  =& \E( \langle \DM (\phi(\Vort(t,\Vort_0)), \Kon \rangle_{L^2(0,t;L_2)})
  = \E\left( \phi(\Vort(t,\Vort_0)) \int_0^t \Kon dW \right).
       \label{eq:NoisePerturbMSG}
\end{align}

With these preliminaries in hand we now return to  \eqref{eq:MSGInitD} and compute
\begin{align*}
\nabla P_t \phi(\Vort_0) \cdot \xi &= 
 \E  \left( \nabla \phi(\Vort(t, \Vort_{0})) \cdot \LinN_{0,t}\Kon \right) + \E \left( \nabla \phi(\Vort(t, \Vort_{0})) \cdot (\LinI_{0,t} \xi - \LinN_{0,t}\Kon)\right)
 \notag\\
&= \E \left( \phi( \Vort(t,\Vort_0)) \int_0^t \Kon  dW \right)  + \E \left( \nabla \phi(\Vort(t,\Vort_0)) \rho(t,\xi,\Kon,\Vort(t,\Vort_0)) \right)
\end{align*}
which holds for any $\Kon \in L^2(\Omega;L^2_{\rm loc}([0,\infty);L_2))$ and  $\rho$ is the solution of the ``control problem''
\begin{align}
& \partial_t \rho + \Lambda^\gamma \rho + \nabla B (\Vort(t,\Vort_0) ) \rho = - \sigma \Kon \label{eq:ASF:control}\\
& \rho(0) = \xi. \label{eq:ASF:control:IC}
\end{align}

In order to prove Proposition~\ref{thm:ASF:easy} we need a procedure to assign to every $\xi \in H^r$, with $\| \xi\|_{H^r} = 1$ an element  $\Kon^\xi \in L^{2}(\Omega; L^{2}_{loc}([0, \infty); L_{2}))$ such that 
\begin{align}
	\lim_{t\to \infty} \left( \sup_{\| \xi\|_{H^r} = 1} 
	\E \| \rho(t,\xi,\Kon^\xi,\Vort(t,\Vort_0)) \|_{H^r} \right) 
	=: \lim_{t\to \infty} \delta(t) = 0
	\label{eq:ASF:control:2}
\end{align}
and
\begin{align}
	\sup_{t \geq 0, \|\xi\|_{H^r} = 1} \E  \left| \int_0^t \Kon^\xi(s) dW  \right|
		\leq C < \infty. 
		\label{eq:ASF:control:1}
\end{align}
The propose of Sections~\ref{sec:control:H:-1} and \ref{sec:control:H:r} is to establish \eqref{eq:ASF:control:2}, \eqref{eq:ASF:control:1} for a suitable control $\Kon^\xi$ which we will define next.

For this purpose we introduce the classical projection operators: $P_N$ is the projection from $H^r$ to 
$H_N = {\rm span} \{ e_k \colon k\in{\mathbb Z}_0^2, |k|\leq N\}$; and 
$Q_N = 1 - P_N$ is the projection on the orthogonal complement of $H_N$.
Define $\sigma_{*}: H^r \to L_2$ according to
\begin{align*}
   (\sigma^*\Vort)_k  = 
   \begin{cases}
   \frac{1}{q_k} \langle \Vort, e_k \rangle & k \in \ZZF\\
   0 & \textrm{otherwise}.
   \end{cases}
\end{align*}
As such that $\sigma \sigma_{*}= 1$ on $H_N$  and let 
\begin{align}
\Kon = - \lambda_{N}^{\gamma/2} \sigma_* P_N \rho\label{eq:ASF:control:v:choice}
\end{align}
with $N$ to be determined below.
The above choice of $\Kon$ implies that \eqref{eq:ASF:control} is equivalent to 
\begin{align}
\partial_t \rho + \Lambda^\gamma \rho +   \nabla B(\Vort(t,\Vort_0)) \rho  = - \lambda_{N}^{\gamma/2}  P_N \rho \label{eq:ASF:control:3}
\end{align}

The first observation is that even when attempting to show that $\E \|\rho(t)\|_{L^2} \to 0$ as $t\to \infty$, for $\gamma \in (0,1)$ we obtain an equation of the type
\begin{align*}
\frac{d}{dt} \| \rho\|_{L^2} \leq C \|\rho \|_{L^2} \left(\| \Vort \|_{H^1}^{1+\eps} - \lambda_N^{\gamma/2}  \right) 
\end{align*}
for some small $0 < \eps < 1$. Upon applying the Gr\"onwall inequality and taking expected values, it seems that the desired (exponential) decay on $\rho$ in $L^2$ may be obtained {\em if} we had that
\begin{align}
\E \exp\left( \int_0^T \|\Vort(t) \|_{H^1}^{1+\eps} dt \right) \label{eq:H1:exp:mom}
\end{align}
grows at most exponentially in $T$. The main difficulty is that a bound on the exponential moment of the $H^1$ norm in \eqref{eq:H1:exp:mom} is {\em not available}. Indeed, we are only able to prove exponential moments for the $L^p$ norm of $\Vort$, cf.~Section~\ref{sec:Vort:exp} above. 

To overcome this difficulty, the key step is to first show that the the expected value of the $H^{-1}$ norm of $\rho$ decays, and then bootstrap this information to $H^r$ norms.

\subsection{Estimates on the Control Equation in \texorpdfstring{$H^{-1}$}{H -1}}
\label{sec:control:H:-1}

Let $v = K \ast \rho$ and $u = K \ast \Vort$. Convolving \eqref{eq:ASF:control:3} with $K$, we thus obtain
\begin{align}
\partial_t v + \Lambda^\gamma v + \lambda_{N}^{\gamma/2} P_N v + u \cdot \nabla v + \nabla \pi + v \cdot \nabla u = 0, \qquad \nabla \cdot v = 0,
\label{eq:ASF:control:4}
\end{align}
for some suitable mean-zero pressure $\pi$. Multiplying \eqref{eq:ASF:control:4} by $v$ and integrating over the torus yields
\begin{align}
\frac 12 \frac{d}{dt} \|v\|_{L^2}^2 + \| \Lambda^{\gamma/2} v \|_{L^2}^2 + \lambda_{N}^{\gamma/2} \| P_N v \|_{L^2}^2 = - \int (v \cdot \nabla u) \cdot v \leq \| v\|_{L^{2+\eps}}^2 \|\nabla u\|_{L^{\frac{2+\eps}{\eps}}}
\label{eq:ASF:control:5}
\end{align}
where we let $\eps = \eps(\gamma) = 2\gamma/(6-\gamma)$. This choice of $\eps$, in view of \eqref{eq:sobolev:abstract} gives that
\begin{align}
\| v\|_{L^{2+\eps}}^2 \|\nabla u\|_{L^{\frac{2+\eps}{\eps}}} 
= \| v\|_{L^{\frac{12}{6-\gamma}}}^2 \|\nabla u\|_{\frac{6}{\gamma}}
&\leq C \|\Lambda^{\gamma/6} v \|_{L^2}^2 \|\Vort\|_{L^{\frac{6}{\gamma}}} \notag\\
&\leq C  \|\Lambda^{\gamma/2} v \|_{L^2}^{2/3}  \| v \|_{L^2}^{4/3} \|\Vort\|_{L^{\frac{6}{\gamma}}} \notag\\
& \leq \frac 12 \| \Lambda^{\gamma/2}v\|_{L^2}^2 + C \|v\|_{L^2}^2 \|\Vort\|_{L^\frac{6}{\gamma}}^{3/2} \notag\\
&\leq \frac 12 \| \Lambda^{\gamma/2}v\|_{L^2}^2 + \frac{\kappa_\gamma}{2} \|v\|_{L^2}^2 \|\Vort\|_{L^\frac{6}{\gamma}}^{2} + C \kappa_\gamma^{-3} \|v\|_{L^2}^2
\label{eq:ASF:control:6}
\end{align}
where $0 < \kappa_\gamma  \ll 1$ is to be chosen later. Inserting \eqref{eq:ASF:control:6} into \eqref{eq:ASF:control:5}, along with the standard lower bound for $\|\Lambda^{\gamma/2} Q_N v\|_{L^2}^2$, yields
\begin{align}
\frac{d}{dt} \|v\|_{L^2}^2 +    \lambda_N^{\gamma/2}   \|v \|_{L^2}^2 
&\leq \frac{d}{dt} \|v\|_{L^2}^2 +    \lambda_N^{\gamma/2}   \|P_N v \|_{L^2}^2 + \|\Lambda^{\gamma/2} Q_N v\|_{L^2}^2 \notag\\
&\leq \kappa_\gamma  \| v\|_{L^2}^2 \|\Vort\|_{L^{6/\gamma}}^2 + C \kappa_\gamma^{-3} \|v\|_{L^2}^2
\label{eq:ASF:control:7}
\end{align}
for some universal positive constant $C$.
Assuming $N$ is chosen sufficiently large so that 
\begin{align}
\frac{\lambda_N^{\gamma/2}}{2} \geq C \kappa_\gamma^{-3}
\label{eq:N:cond:1}
\end{align}
where $C$ is the constant in \eqref{eq:ASF:control:7}, we thus obtain
\begin{align}
\frac{d}{dt} \|v\|_{L^2}^2 \leq \left(  \kappa_\gamma  \|\Vort\|_{L^{6/\gamma}}^2 - \frac{\lambda_N^{\gamma/2}}{2} \right) \| v\|_{L^2}^2  \label{eq:ASF:control:8}
\end{align}
where $\kappa_\gamma$ is yet to be chosen. Upon applying the Gr\"onwall inequality and taking expected values, we thus obtain
\begin{align*}
\E \|\rho(T)\|_{H^{-1}}^2 = \E \|v(t)\|_{L^2}^2  \leq \|\xi\|_{H^{-1}}^2 \E \exp \left(\kappa_\gamma \int_0^T \|\Vort(t)\|_{L^{6/\gamma}}^2 dt -  \frac{\lambda_N^{\gamma/2}T }{2} \right) 
\end{align*}
for any $T>0$. Now, we choose 
\begin{align}
\kappa_\gamma = \frac{1}{C(1 + \|\sigma\|_{\LL^{6/\gamma}}^4)} 
\label{eq:kappa:gamma:def}
\end{align}
to be sufficiently small so that \eqref{eq:expo:mom:II} holds with $ \kappa = \kappa_\gamma^{1/2}$ (see also \eqref{eq:eps:kappa} with $\eps=\kappa=\kappa_\gamma^{1/2}$). This allows us to apply the estimate \eqref{eq:expo:mom:II} with $p = 6/\gamma$, and obtain that
\begin{align*}
\E \|\rho(T)\|_{H^{-1}}^2 \leq \|\xi\|_{H^{-1}}^2  \exp\left(T -  \frac{\lambda_N^{\gamma/2}T }{2} \right) C  \E \exp\left( \kappa_\gamma^{1/2} \|\Vort_0\|_{L^{6/\gamma}}^2 \right).
\end{align*}
Therefore, if we ensure that $N$ is sufficiently large so that
\begin{align}
\frac{\lambda_N^{\gamma/2}}{4} \geq 1 \vee C\poly(\|\sigma\|_{L^{6/\gamma}})\label{eq:N:cond:2}
\end{align}
we have established the exponential decay of the $H^{-1}$ norm of the control
\begin{align}
\E \|\rho(T)\|_{H^{-1}}^2 \leq \|\xi\|_{H^{-1}}^2  \exp\left(-  \frac{\lambda_N^{\gamma/2}T }{4} \right) C  \E \exp\left(\kappa_\gamma^{1/2} \|\Vort_0\|_{L^{6/\gamma}}^2   \right)
\label{eq:ASF:control:9}
\end{align}
where $\kappa_\gamma$ is a sufficiently small constant, that depends on $\gamma$ and  on  $\| \sigma\|_{\LL^{6/\gamma}}$.

\subsection{Estimates on the Control in \texorpdfstring{$H^{r}$}{Hr}}
\label{sec:control:H:r}

Next, we bootstrap the decay obtained in \eqref{eq:ASF:control:9} for the $H^{-1}$ norm, to a decay for the $H^r$ norm of the control.
As in Section~\ref{sec:Vort:Hr} we need to appeal to the smoothing effect encoded in the equations.
This time we consider
\begin{align} 
s(t)  = 
\begin{cases} r - 1 + t \gamma  , & t \in [0, T_\gamma],\\
r , & t > T_\gamma,
\end{cases}
\label{eq:alpha:control}
\end{align}
where we let $T_\gamma = \gamma^{-1}$. Note that for $x>0$ and $\gamma>0$ we have $4 \gamma \log x \leq x^{2\gamma}$ which is the reason why we let $\dot{s}(t) = \gamma $ on $[0,T_\gamma]$. More precisely, this choice of slope in $s(t)$ yields
\begin{align*} 
2 \dot{s}(t) \| (\log \Lambda)^{1/2} \Lambda^{s(t)} \rho \|_{L^2}^2 \leq \frac 12 \|\Lambda^{s(t) + \gamma/2} \rho \|_{L^2}^2.
\end{align*}
In view of the above discussion, the $H^{s(t)}$ energy estimate for the control equation yields
\begin{align} 
\frac 12 \frac{d}{dt} \|\Lambda^s \rho\|_{L^2}^2 &+ \frac 14 \|\Lambda^{s + \gamma/2} \rho\|_{L^2}^2 + \frac{\lambda_N^{\gamma/2}}{4} \|\Lambda^s \rho \|_{L^2}^2 \notag\\
&\leq \left| \int \Lambda^s B(\rho,\Vort) \Lambda^s \rho dx \right| + \left| \int \Lambda^s B(\Vort,\rho) \Lambda^s \rho dx \right| =: T_1 + T_2.
\label{eq:ASF:Hs:1}
\end{align}

Note that $s(t) \geq r- 1 > 1$. Thus, $H^s$ is an algebra, and we have a direct bound for $T_1$ as
\begin{align}
T_1 &\leq C \|\Lambda^{s} \rho\|_{L^2} \|K \ast \Lambda^{s}  \rho \|_{L^2} \| \nabla \Vort\|_{H^s}  \leq   C \|\Vort\|_{H^{s+1}} \|\Lambda^s \rho\|_{L^2}^2.
\label{eq:ASF:T1}
\end{align}
Here we also used the Poincar\'e inequality.
To estimate the $T_2$ term in \eqref{eq:ASF:Hs:1} we note that $\int B(\Vort,\Lambda^s \rho) \Lambda^s\rho dx =0$, and appeal to the commutator
\eqref{eq:comm:abstract}. The Sobolev embedding \eqref{eq:sobolev:abstract}, and the Poincar\'e inequality, letting $0< \eps \ll 1$ we get
\begin{align} 
T_2 
&\leq C \|\Lambda^{s-1} \Vort\|_{L^{\frac{4+2\eps}{\eps}}} \|\nabla \rho\|_{L^{2+\eps}} \|\Lambda^s \rho\|_{L^2} + C \|\Vort\|_{L^{\frac{8}{\gamma}}} \|\Lambda^{s} \rho\|_{L^{\frac{8}{4-\gamma}}} \|\Lambda^s\rho \|_{L^2} \notag\\
&\leq C \|\Lambda^s\Vort\|_{L^2} \|\Lambda^{s} \rho\|_{L^{2}}^2 + C \|\Vort\|_{H^1} \|\Lambda^{s+\gamma/4} \rho\|_{L^2} \|\Lambda^s \rho\|_{L^2} \notag\\
&\leq \frac{1}{8} \|\Lambda^{s+\gamma/2}\rho \|_{L^2}^2 + C \left( \|\Lambda^s\Vort\|_{L^2} + \| \Vort\|_{H^1}^2 \right) \| \Lambda^s \rho\|_{L^2}^{2}.
\label{eq:ASF:T2}
\end{align}
Therefore, inserting the estimates \eqref{eq:ASF:T1}--\eqref{eq:ASF:T2} into \eqref{eq:ASF:Hs:1} and appealing to the Poincar\'e inequality we obtain
\begin{align} 
 \frac{d}{dt} \|\Lambda^s \rho\|_{L^2}^2  + \frac 18 \|\Lambda^{s + \gamma/2} \rho\|_{L^2}^2 +& \frac{\lambda_N^{\gamma/2}}{2} \|\Lambda^s \rho \|_{L^2}^2 \notag\\
 &\leq C \left(1 +  \|\Vort\|_{H^{s+1}}^2 \right) \|\Lambda^{s}\rho\|_{L^2}^{2}  \notag \\
 & \leq C \left(1 +  \|\Vort\|_{H^{s+1}}^2 \right) \|\Lambda^{-1}\rho\|_{L^2}^{\delta} \|\Lambda^s \rho\|_{L^2}^{\delta}  \|\Lambda^{s+\gamma/2} \rho\|_{L^2}^{2-2\delta} \notag\\
 &\leq \frac 18 \|\Lambda^{s + \gamma/2} \rho\|_{L^2}^2  + C (1 + \|\Vort\|_{H^{s+1}}^{2/\delta} ) \|\Lambda^{-1} \rho\|_{L^2} \|\Lambda^s \rho\|_{L^2}
\label{eq:ASF:Hs:2}
\end{align}
where
\begin{align*}
\delta = \frac{\gamma}{s+1+\gamma}.
\end{align*}
Note that by \eqref{eq:alpha:control} we have $\gamma/(r+1+\gamma) \leq  \delta \leq \gamma/(r+\gamma)$.
After canceling the dissipative terms, we divide both sides of \eqref{eq:ASF:Hs:2} by $\|\Lambda^s \rho\|_{L^2} $ and obtain
\begin{align} 
&\frac{d}{dt} \|\Lambda^s \rho\|_{L^2}  + \frac{ \lambda_N^{\gamma/2}}{4} \|\Lambda^s \rho \|_{L^2}    \leq C \left(1 +     \| \Vort\|_{H^{s+1}}^{\frac{2(r+1+\gamma)}{\gamma}} \right) \|\Lambda^{-1}\rho\|_{L^2} 
\label{eq:ASF:Hs:4}.
\end{align}
Using Duhamel's formula, and taking expected values, we thus obtain that 
\begin{align} 
& \E \| \Lambda^{s(t)} \rho(t)\|_{L^2}  \notag\\
&\leq \| \Lambda^{r-1} \xi \|_{L^2}  \exp\left(-\frac{t   \lambda_N^{\gamma/2}}{4} \right)\notag\\ 
&\ + C \E \int_0^t \exp\left(-\frac{(t-\tau) \lambda_N^{\gamma/2}}{4} \right) \|\rho(\tau)\|_{H^{-1}}    \left(1 +    \|\Lambda^{s(\tau)+1}\Vort(\tau)\|_{L^{2}}^{\frac{2(r+1+\gamma)}{\gamma} } \right) d\tau\notag\\
&\leq \| \Lambda^{r-1} \xi \|_{L^2} \exp\left(-\frac{t \lambda_N^{\gamma/2}}{4} \right)  \notag\\
&\ + C \left( \E \int_0^t \exp\left(-  \frac{(t-\tau) \lambda_N^{\gamma/2}}{2} \right) \|\rho(\tau)\|_{H^{-1}}^{2} d\tau\right)^{\frac{1}{2}}   \left( \E \int_0^t 1 + \|\Lambda^{s(\tau)+1}\Vort(\tau)\|_{L^{2}}^{\frac{4(r+1+\gamma)}{\gamma} }  d\tau \right)^{\frac{1}{2}}.
\label{eq:ASF:Hs:5}
\end{align}
To conclude, we use estimate \eqref{eq:ASF:control:9} which gives us exponential decay of $\E \|\rho(\tau) \|_{H^{-1}}^2$,  and estimate \eqref{eq:KB:moment:3}, which gives us a control\footnote{Note that the $s(t)$ in \eqref{eq:KB:moment:3} is not the same as the $s(t)$ in \eqref{eq:alpha:control}. The former is larger by $1$ than the latter.}
 of $\int_0^t \E \|\Lambda^{s(\tau)+1}  \Vort(\tau)\|_{L^2}^q d\tau$, with any $q\geq 2$. Therefore, from \eqref{eq:KB:moment:3}, \eqref{eq:ASF:control:9}, and \eqref{eq:ASF:Hs:5} we obtain
 \begin{align} 
&\E\| \Lambda^{s(t)} \rho(t)\|_{L^2}  \notag\\
&\quad \leq \| \Lambda^{r-1} \xi \|_{L^2} \exp\left(-\frac{t \lambda_N^{\gamma/2}}{4} \right) \notag\\
&\qquad + C \|\xi\|_{H^{-1}}  \exp\left(-\frac{t \lambda_N^{\gamma/2}}{8} \right)  \Bigl(  \expo(\kappa_\gamma, \|\Vort_0\|_{L^{6/\gamma}}) \Bigr)^{1/2}   \Bigl(  \poly(\|\Vort_0\|_{H^r}) + C t  \poly(\|\sigma\|_{\HH^{r+2}}) \Bigr)^{1/2}.
\label{eq:ASF:Hs:6}
\end{align}
The degree of the polynomial $\PP$ above may be computed explicitly\footnote{In \eqref{eq:ASF:Hs:6}, we have $\PP(x) = 1+ x^n$, where $n$ is the smallest integer larger than $\frac{16 (r+1+\gamma) ( (4+\gamma) (r+2+\gamma) - 4)(1+\gamma)}{\gamma^3 (6+\gamma)}$, but this explicit value is not important.} solely in terms of  $r$ and $\gamma$. The coefficient $\kappa_\gamma$ in the exponential function depends on $\gamma$ and $\|\sigma\|_{\LL^{6/\gamma}}$. Here we also used that the initial data $\Vort_0$ is in fact deterministic.

To conclude the argument, we need to wait let $t \geq T_\gamma = \gamma^{-1}$, so that $s(t) = r$,  and obtain
\begin{align} 
\E \|\rho(t)\|_{H^r} 
&\leq C \|\xi\|_{H^{r-1}} \exp\left(-\frac{t \lambda_N^{\gamma/2}}{8} \right)   \Bigl(   \expo(\kappa_\gamma, \|\Vort_0\|_{L^{6/\gamma}}) \Bigr)^{1/2}   \Bigl(    \poly(\|\Vort_0\|_{H^r}) +  t  \poly(\|\sigma\|_{\HH^{r+2}}) \Bigr)^{1/2}  \notag\\
&\leq C \|\xi\|_{H^{r-1}} \exp\left(-\frac{t \lambda_N^{\gamma/2}}{8} \right)  (1+t)^{1/2}    \expo(\kappa_\gamma, \|\Vort_0\|_{L^{6/\gamma}})     \poly(\|\Vort_0\|_{H^r} , \|\sigma\|_{\HH^{r+2}})  \notag\\
&\leq C \|\xi\|_{H^{r-1}} \exp\left(-\frac{t \lambda_N^{\gamma/2}}{16} \right)      \expo(\kappa_\gamma, \|\Vort_0\|_{L^{6/\gamma}})  \poly(\|\Vort_0\|_{H^r} , \|\sigma\|_{\HH^{r+2}})  
\label{eq:ASF:decay:final}
\end{align}
whenever $N$ is sufficiently large so that \eqref{eq:N:cond:1} and \eqref{eq:N:cond:2} holds. 

\subsection{Estimates on the Stochastic Integral Term}

It finally remains to verify \eqref{eq:ASF:control:1} for the choice
$\Kon$ given in \eqref{eq:ASF:control:v:choice}.  Let us first note that $\Kon$ is adapted. As such, with the It\=o isometry we obtain that 
\begin{align}
\E \left|\int_0^t \Kon(s) dW_s \right|^2 
&= \E \int_0^t|\lambda_{N}^{\gamma/2} \sigma_*  P_N \rho(s) |_{L_2}^2 ds\notag\\
&\leq  \lambda_{N}^{\gamma} \|\sigma_*\|_{{\mathcal L}(H^r,L_2)}^2 \E \int_0^\infty \|P_N \rho\|_{H^r}^2 ds \notag\\
& \leq  \lambda_{N}^{2r + 2+\gamma} \|\sigma_* \|_{{\mathcal L}(H^r,L_2)}^2 \E \int_0^\infty \| \rho\|_{H^{-1}}^2 ds
\label{eq:ASF:SIT:1}
\end{align}
To prove \eqref{eq:ASF:control:1}, we combine \eqref{eq:ASF:SIT:1} with \eqref{eq:ASF:control:9} and obtain
\begin{align} 
\sup_{t \geq 0, \|\xi\|_{H^r} = 1} \E  \left| \int_0^t \Kon^\xi(s) dW_s  \right| 
&\leq C \lambda_{N}^{2r + 2+\gamma/2} \|\sigma_* \|_{{\mathcal L}(H^r,L_2)}^2   \E \expo(\kappa_\gamma^{1/2}, \|\Vort_0\|_{L^{6/\gamma}}^2) \notag\\
&\leq C \frac{\PP(\|\sigma\|_{\LL^{6/\gamma}}) }{\min_{k \in \ZZF} |q_k|^2} \E \expo(\kappa_\gamma^{1/2}, \|\Vort_0\|_{L^{6/\gamma}}^2) \label{eq:ASF:SIT:2}
\end{align}
where we have used that $\Lambda_N$ is chosen to proportional to a polynomial in $\|\sigma\|_{\LL^{6/\gamma}}$.
The emphasis here is that the above bound is independent of $t$ and $\xi$.

\begin{remark}
\label{rmk:ASF:Girsanov:Thm}
A different approach may be used to establish the asymptotic strong Feller property, 
which does not require exponential moment estimates nor the use of Mallivan Calculus, but which retains some of the spirit of the above estimates.
Here one couples nearby solutions using the Girsanov theorem and the Foias-Prodi estimates (see, e.g.~\cite{KuksinShirikian12}). Unfortunately, this framework appears to be ill-suited to establish ergodic properties and mixing in the hypoelliptic forcing regime.
\end{remark}

\section{Weak Irreducibility}
\label{sec:WI}
\setcounter{equation}{0}

Let us denote by $\BB(R)$ the ball of radius $R$ about the origin in $H^r$.

\begin{proposition}[\bf Weak Irreducibility]\label{prop:WI}
Let $\mu \in \Pr(H^r)$ be invariant for $\{ P_t \}_{t \geq 0}$. Then, for any $\eps>0$, we have $\mu (\BB(\eps)) >0$, 
i.e. $0 \in {\rm supp}(\mu)$.
\end{proposition}

\begin{proof}[Proof of Proposition~\ref{prop:WI}]
The proof is based on establishing two properties.
First, we show that there exists $\lambda>0$, and $\delta_1>0$ such that 
\begin{align}
\mu \left(\BB(\lambda) \right) \geq \delta_1 \label{eq:WI:i}
\end{align}
for every invariant measure $\mu$. Secondly, we prove that
for any $R>0$, $\eta>0$, there exist $T = T(R,\eta)>0$ and $\delta_2=\delta_2(R,\eta)>0$ such that 
\begin{align}
\inf_{\|\Vort_0\|_{H^r} \leq R} P_T(w_0, \BB(\eta)) \geq \delta_2 \label{eq:WI:ii}
\end{align}
where we recall that $P_T(\Vort_0, \BB(\eta)) = \Prb( \| \Vort(T,\Vort_0)\|_{H^r} \leq \eta)$.

To see that \eqref{eq:WI:i} and \eqref{eq:WI:ii} give the proof of the proposition, let $\mu \in \Pr(H^r)$ be invariant, and hence 
\begin{align*}
\mu(A) = P_t^\ast \mu(A) = \int_{H^r} P_t(\Vort_0, A) d\mu(\Vort_0)
\end{align*}
for any $A$ in the Borelians on $H^r$.
For $\eta >0$ arbitrary,  let $T$ such that \eqref{eq:WI:ii} holds, and let $\lambda$ be as in \eqref{eq:WI:i}. Thus we obtain by letting $R=\lambda$  that
\begin{align*}
\mu(\BB(\eta)) = \int_{H^r} P_T(\Vort_0, \BB(\eta)) d\mu(\Vort_0) \geq \int_{\BB(\lambda)}P_T(\Vort_0, \BB(\eta)) d\mu(\Vort_0) \geq \delta_1 \delta_2 >0.
\end{align*}
It thus remains to establish \eqref{eq:WI:i} and \eqref{eq:WI:ii}. 

In order to prove \eqref{eq:WI:i}, let $\mu$ be invariant, and let $\Vort_S$ be an associated stationary solution of \eqref{eq:frac:SNSE}. From the estimate \eqref{eq:KB:moment:3:stat}, and the Poincar\'e inequality we conclude that
\begin{align*}
\E \|\Vort_S\|_{H^{r}}^2 \leq C \PP(\| \sigma\|_{\HH^r}) <\infty.
\end{align*} 
Therefore, we have 
\begin{align*}
\mu(\BB(\lambda)^c) \leq \frac{1}{\lambda^2} \int_{H^r} \| \Vort\|_{H^r}^2 d \mu(\Vort) = \frac{1}{\lambda^2} \E \|\Vort_S\|_{H^r}^2 \leq \frac{C \PP(\|\sigma\|_{\HH^r})}{\lambda^2},
\end{align*}
and letting $\lambda$ be sufficiently large (independently of $\mu$), we obtain \eqref{eq:WI:i} with $\delta_1 = 1/2$, for example.

To establish \eqref{eq:WI:ii}, we consider the Ornstein-Uhlenbeck process $Z$ given by
\begin{align}
d Z + \Lambda^\gamma Z dt = \sigma dW, \quad Z(0) = 0 \label{eq:Z}
\end{align}
and consider the change of variables
\begin{align*}
\bar \Vort = \Vort - Z
\end{align*}
that obeys the PDE with random coefficients
\begin{align}
\partial_t \bar \Vort + \Lambda^\gamma \bar \Vort + B(\bar \Vort + Z,\bar \Vort) + B(\bar \Vort + Z, Z) = 0, \quad \bar \Vort(0) = \Vort_0.
\label{eq:shifted:SNSE}
\end{align}
For $\delta,T>0$, we introduce the set 
\begin{align}
\Omega_{\delta,T} = \left\{ w \in \Omega \colon | W_s^j (w) | \leq \delta, \mbox{ for all } s\in [0,T] \mbox{ and all } j \in  \ZZF \right\}.
\label{eq:Omega:dT}
\end{align}
Using standard properties of Brownian motion, since $|\ZZF|<\infty$ we know that for any $\delta,T>0$, there exists $\delta_2 = \delta_2(\delta,T)>0$ such that
\begin{align}
\Prb( \Omega_{\delta,T}) \geq \delta_2 > 0.
\label{eq:prb:Omega:dT}
\end{align}
On the set where the Brownian motions stay close to the origin, one may use the representation of $Z$ as a stochastic convolution to establish the classical fact.
\begin{proposition} \label{prop:Z:small}
Let $r>2$, $\gamma>0$, and $|\ZZF| < \infty$. For any $\delta, T >0$, there exists a deterministic constant $\eps_{\delta,T}  >0$ such that 
$\eps_{\delta,T} \to 0$ as $\delta \to 0$ for $T$ fixed, and such that
\begin{align}
\sup_{t \in [0,T]} \| Z(t,w) \|_{H^{r+2}} \leq \eps_{\delta,T}, \quad \mbox{ for all } w \in \Omega_{\delta,T},
\label{eq:Z:small}
\end{align}
where $\Omega_{\delta,T}$ is as defined in \eqref{eq:Omega:dT}.
\end{proposition}
Proposition \eqref{prop:Z:small} implies that for trajectories starting in $\Omega_{\delta,T}$, the coefficients of the nonlinear PDE \eqref{eq:shifted:SNSE} are small in $H^{r+1}$. Therefore, using the decay in time given by the dissipative operator $\Lambda^\gamma$, we may expect that after waiting a sufficient amount of time, the shifted vorticity $\bar \Vort$ is also small. More precisely, we have:
\begin{proposition}[\bf Decay for the shifted equations] \label{prop:decay:shifted}
Let $r>2,\gamma>0$, and $R,\eta>0$ be arbitrary. Then there exist $\delta,T >0$ such that for any $\Vort_0 \in \BB(R)$,
\begin{align}
\|\bar \Vort(T,\Vort_0)\|_{H^r} \leq \eta/2 \quad \mbox{on } \Omega_{\delta,T},
\label{eq:decay:shifted}
\end{align}
where $\Omega_{\delta,T}$ is as in \eqref{eq:Omega:dT}. 
\end{proposition}
Assuming Proposition~\ref{prop:decay:shifted} holds, we may now easily complete the proof of \eqref{eq:WI:ii}. Indeed, for $R,\eta>0$ given, we may find $T=T(R,\eta)$ sufficiently large, and $\delta = \delta(R,\eta)$ sufficiently small such that \eqref{eq:decay:shifted} holds. 
In addition, since $\eps_{\delta,T} \to 0$ as $\delta \to 0$, upon possibly further shrinking $\delta$ we can ensure that in \eqref{eq:Z:small} we have $\eps_{\delta,T} \leq \eta/2$.
Then, it follows that
\begin{align}
\inf_{\|\Vort_0\|_{H^r} \leq R} \Prb( \| \Vort(t,\Vort_0)\|_{H^r} \leq \eta)
&\geq  \inf_{\|\Vort_0\|_{H^r} \leq R} \Prb( \|\bar \Vort(t,\Vort_0)\|_{H^r} + \| Z(t) \|_{H^r} \leq \eta) \geq \Prb(\Omega_{\delta,T}) \geq \delta_2
\end{align}
by using \eqref{eq:prb:Omega:dT}. This concludes the prof of Proposition~\ref{prop:WI}, modulo the proof of Proposition~\ref{prop:decay:shifted}, which we establish next.
\end{proof}

\begin{proof}[Proof of Proposition~\ref{prop:decay:shifted}] Fix some $\Vort_0 \in \BB(R)$. Throughout this proof we will work pathwise on the set $\Omega_{\delta,T}$,  where $\delta$ and $T$ will be chosen suitably at the end of the proof. This ensures in view of Proposition~\ref{prop:Z:small} that for $w \in \Omega_{\delta,T}$ we have $\| Z(\cdot,w) \|_{L^\infty(0,T;H^{r+2})} \leq \eps =\eps_{\delta,T}$.

We first obtain a decay estimate on for high $L^p$ norms of the shifted vorticity $\bar \Vort$.  Let $p\geq 2$ be even. Multiplying \eqref{eq:shifted:SNSE} with $\bar \Vort^{p-1}$, integrating over $\TT$, and using \eqref{eq:CC:lower:bound} we obtain
\begin{align} 
\frac{d}{dt} \|\bar \Vort\|_{L^{p}}^{p} + \frac{1}{C_\gamma} \|  \bar \Vort \|_{L^p}^p 
&\leq C \| K \ast \bar \Vort\|_{L^{p}} \|\nabla Z\|_{L^\infty} \|\bar \Vort\|_{L^{p}}^{p-1} + C \|K\ast Z\|_{L^\infty} \|\nabla Z\|_{L^{p}} \|\bar \Vort\|_{L^{p}}^{p-1}\notag\\
&\leq C \eps \|\bar \Vort\|_{L^{p}}^{p} + C \eps^2 \|\bar \Vort\|_{L^{p}}^{p-1} \leq 2C \eps \|\bar \Vort\|_{L^{p}}^{p} +  C\eps^{p+1} 
\label{eq:w:bar:Lp:decay:1}
\end{align}
with $C_\gamma \geq 1$. Therefore, 
if $\delta$ is chosen so that
\begin{align} 
2 C \eps \leq \frac{1}{2 C_\gamma}, \label{eq:delta:cond:2}
\end{align}
we obtain from \eqref{eq:w:bar:Lp:decay:1} that
\begin{align} 
\frac{d}{dt} \|\bar \Vort\|_{L^{p}}^{p} + \frac{1}{2C_\gamma} \|\bar \Vort \|_{L^{p}}^{p} \leq  \eps^{p} 
\label{eq:w:bar:Lp:decay:2}.
\end{align}
Gr\"onwall and  \eqref{eq:w:bar:Lp:decay:2} thus yield
\begin{align} 
\| \bar \Vort(t)\|_{L^{p}}^{p} 
&\leq \|\Vort_0\|_{L^{p}}^{p} \exp \left( -\frac{t}{2C_\gamma}  \right)+  2 C(p) \eps^p . \label{eq:w:bar:Lp:decay}
\end{align}

Next, similarly to Section~\ref{sec:Vort:H1}, we multiply \eqref{eq:shifted:SNSE} with $\Delta \bar \Vort$ and integrate over $\TT$ to obtain
\begin{align} 
\frac{1}{2} \frac{d}{dt} \| \nabla \bar \Vort\|_{L^2}^{2} + \|\Lambda^{\gamma/2} \nabla \bar \Vort \|_{L^2}^2 
&= \int  B(\bar \Vort,\bar \Vort) \Delta \bar \Vort  + \int B(Z,\bar \Vort) \Delta \bar \Vort + \int B(\bar \Vort, Z) \Delta \bar \Vort + \int B(Z,Z) \Delta\bar \Vort \notag\\
&:= T_1 + T_2 +T_3 + T_4.\label{eq:w:bar:H1:decay:1}
\end{align}
Similarly to \eqref{eq:KB:step2:non:2} we estimate
\begin{align*}
|T_1| \leq \int |\nabla K \ast \bar \Vort| |\nabla \bar \Vort|^2 \leq \frac 12 \|\Lambda^{\gamma/2} \nabla \bar \Vort\|_{L^2}^2 + C  \|\bar \Vort\|_{L^{p_\gamma}}^{ p_\gamma},
\end{align*}
where 
\begin{align}
p_\gamma =  4 + \frac{4}{\gamma}.
\label{eq:p:gamma:2}
\end{align}
On the other hand, upon integrating by parts, using the Poincar\'e inequality, the Sobolev embedding, and estimate \eqref{eq:Z:small}, we obtain
\begin{align*} 
|T_2| + |T_3| + |T_4| \leq C \eps \|\nabla \bar \Vort\|_{L^2}^2 + C \eps^2 \|\nabla \bar \Vort\|_{L^2} \leq 2 C \eps \|\nabla \bar \Vort\|_{L^2}^2 + C \eps^3.
\end{align*}
Combining the above estimates yields
\begin{align} 
\frac{d}{dt} \| \nabla \bar \Vort\|_{L^2}^{2} + \| \nabla \bar \Vort \|_{L^2}^2 \leq C \eps \|\nabla \bar \Vort\|_{L^2}^2 + C \eps^3 + C \|\bar \Vort\|_{L^{p_\gamma}}^{p_\gamma} \label{eq:w:bar:H1:decay:2}
\end{align}
for some positive constant $C$. Thus, if we assume that $\eps$ obeys \eqref{eq:delta:cond:2}, with possibly a larger universal constant $C$, we obtain from \eqref{eq:w:bar:Lp:decay} and \eqref{eq:w:bar:H1:decay:2} that
\begin{align} 
\|\nabla \bar \Vort(t) \|_{L^2}^2 
&\leq \|\nabla \Vort_0\|_{L^2}^2 \exp \left(- \frac t 2\right) + C \eps^2 +C(p_\gamma)   \| \Vort_0\|_{L^{p_\gamma}}^{p_\gamma}  \exp\left( - \frac{t}{2C_\gamma} \right)\notag\\
&\leq C(p_\gamma) (1+R^{p_\gamma}) \exp\left( - \frac{t}{2 C_\gamma} \right) + C \eps^2
\label{eq:w:bar:H1:decay}
\end{align}
for all $t\geq 0$, since $\Vort_0$ lies in the ball of radius $R$ around the origin in $H^r$.

We complete the proof of the proposition using estimates that are similar to those in Section~\ref{sec:Vort:Hr}. Taking an $L^2$ inner product of \eqref{eq:shifted:SNSE} with $\Lambda^{2r} \bar \Vort$, and using a commutator estimate for the term corresponding to $T_2$, we obtain
\begin{align} 
\frac{1}{2} \frac{d}{dt} \|\bar \Vort\|_{H^r}^2 + \frac 34 \|\bar \Vort\|_{H^{r+\gamma/2}}^2
&\leq \left| \int [\Lambda^r, (K\ast \bar \Vort)\cdot \nabla]\bar \Vort \Lambda^r \bar \Vort \right| + C (\eps+\eps^2) \|\bar \Vort\|_{H^r}^2 + C \eps^2 \|\bar \Vort\|_{H^r},
\label{eq:w:bar:Hr:decay:1}
\end{align}
where $[\cdot,\cdot]$ denotes the usual commutator, and we have appealed to \eqref{eq:Hs:product}, the Poincar\'e inequality, the Sobolev embedding, and estimate \eqref{eq:Z:small}. Using the commutator estimate in Lemma~\ref{lem:commutator}, we conclude that there exists 
$ q = q(r,\gamma) \geq 2$, such that \eqref{eq:w:bar:Hr:decay:1} becomes
\begin{align} 
\frac{d}{dt} \|\bar \Vort\|_{H^r}^2 + \|\bar \Vort\|_{H^{r+\gamma/2}}^2 \leq C \|\bar \Vort\|_{H^1}^q + C \eps \|\bar \Vort\|_{H^r}^2 + C \eps^3
\label{eq:w:bar:Hr:decay:2}.
\end{align}
Again, in view of the smallness condition \eqref{eq:delta:cond:2} on $\eps$, with a possibly larger constant $C$, we conclude from \eqref{eq:w:bar:Hr:decay:2} that 
\begin{align} 
\|\bar \Vort(t) \|_{H^r}^2 
&\leq \| \Vort_0 \|_{H^r}^2 \exp\left(-\frac t2\right) + C \eps^3 + C (1+ R^{p_\gamma})^{q/2} \int_0^t \exp\left(-\frac{t-s}{2} \right) \exp\left(-\frac{s q}{4C(p_\gamma)} \right) ds \notag\\
&\leq C (1 + R^{p_\gamma})^{q/2} \exp\left(-\frac{tq}{4C_\gamma} \right) + C \eps^3.
\label{eq:w:bar:Hr:decay}
\end{align}
where $q$ is as given in Lemma~\ref{lem:commutator}, and $p_\gamma$ is given by \eqref{eq:p:gamma:2}.

The proof of \eqref{eq:decay:shifted} is now complete by letting $\delta$ be sufficiently small such that  \eqref{eq:delta:cond:2}, and $C \eps_{\delta,T}^3 \leq \eta/4$ hold, and then letting $T$ be large enough so that $C (1+R^{p_\gamma})^{q/2} \exp(-Tq/4C_\gamma) \leq \eta/4$.
\end{proof}

\appendix

\section{Lower bound for the fractional Laplacian in \texorpdfstring{$L^p$}{Lp}}
\setcounter{equation}{0}
\label{app:Poincare}

Let ${\mathbb{T}}^d = (-\pi,\pi]^d$, and let $\theta(x)$ be a smooth enough scalar, and have zero mean, that is $\int_{{\mathbb{T}}^d} \theta(x) dx = 0$. 
We recall (see e.g. \cite{CordobaCordoba2004,RoncalStinga12} and references therein) the definition of the fractional Laplacian on the torus. For $\gamma \in (0,2)$ we have
\begin{align*}
\Lambda^\gamma \theta(x) = P.V. \int_{{\mathbb{T}}^d} \left(  \theta (x)- \theta (y) \right) K_\gamma(x-y) dy
\end{align*}
where for $z\neq 0$ the kernel $K_\gamma$ is defined as
\begin{align*}
K_\gamma(z) = c_{d,\gamma} \sum_{k \in \ZZ^d} \frac{1}{|z- 2\pi k|^{d+\gamma}}
\end{align*}
and the normalization constant is 
\begin{align*}
c_{d,\gamma} = \frac{2^\gamma \Gamma( (n+\gamma)/2)}{|\Gamma(-\gamma/2)| \pi^{d/2}}.
\end{align*}

Let $p\geq 2$ be even. The goal of this appendix is to prove:
\begin{proposition}[\bf Fractional $L^p$ Poincar\'e]\label{prop:Poincare}
Let $p\geq 2$ be even, $0 \leq \gamma \leq 2$, and let $\theta$ have zero mean on ${\mathbb T}^d$, where $d\geq 1$. Then 
\begin{align}
\int_{{\mathbb{T}}^d} \theta^{p-1}(x) \Lambda^\gamma \theta(x) dx \geq \frac{1}{C_{d,\gamma}} \|\theta\|_{L^p}^p + \frac{1}{p} \| \Lambda^{\gamma/2}  ( \theta^{p/2}) \|_{L^2}^2 
\label{eq:Poincare}
\end{align}
holds, with an explicit constant $C_{d,\gamma} \geq 1$ given by \eqref{eq:poincare:constant} below. 
\end{proposition}

\begin{proof}[Proof of Proposition~\ref{prop:Poincare}]
Of course, unless $\theta$ has zero mean, we cannot expect \eqref{eq:Poincare} to hold, as can be seen by letting $\theta=1$. Also, when $p=2$, inequality \eqref{eq:Poincare} trivially holds (it's just the Poincar\'e inequality) by staring at the Fourier series. Hence for the rest of the proof we let $p \geq 4$ be even. Lastly, the case $\gamma=0$ trivially holds, while the case $\gamma=2$ is follows upon integration by parts.

For $0<\gamma<2$ we have 
\begin{align}
& \int \theta^{p-1}(x) \Lambda^\gamma \theta(x) dx \notag\\
&= P.V. \intint \theta^{p-1}(x) \left( \theta(x) - \theta(y) \right) K_\gamma(x-y) dy dx\notag\\
&= \frac 12 P.V. \intint \left(\theta^{p-1}(x) -\theta^{p-1}(y) \right) \left( \theta(x) - \theta(y) \right) K_\gamma(x-y) dy dx \notag\\
&= \frac{1}{2p} P.V. \intint \left( p\left(\theta^{p-1}(x) -\theta^{p-1}(y) \right) \left( \theta(x) - \theta(y) \right)  - 2 \left( \theta^{p/2}(x) -\theta^{p/2}(y) \right)^2  \right) K_\gamma(x-y) dy dx \notag\\
&\qquad + \frac 1p P.V. \intint \left( \theta^{p/2}(x) -\theta^{p/2}(y) \right)^2   K_\gamma(x-y) dy dx \notag\\
&= \frac{1}{2p} P.V. \intint f_p(\theta(x),\theta(y)) K_\gamma(x-y) dy dx + \frac{1}{p} \| \Lambda^{\gamma/2}  ( \theta^{p/2}) \|_{L^2}^2 
=: \frac{1}{2p} {\mathcal T} +  \frac{1}{p} \| \Lambda^{\gamma/2}  ( \theta^{p/2}) \|_{L^2}^2 
\label{eq:identity}
\end{align}
where the double integral is over ${\mathbb{T}}^{2d}$, and we have defined
\begin{align*}
f_p(a,b) = p(a^{p-1} - b^{p-1}) (a-b) -2 (a^{p/2} - b^{p/2})^2.
\end{align*}
it can be easily seen that $f_p(a,b) \geq 0$ on $\RR^2$ when $p$ is even, and so the term ${\mathcal T}$ is positive. Usually the term ${\mathcal T}$ is dropped in establishing lower bounds. The {trick} is that exactly ${\mathcal T}$ gives the lower bound \eqref{eq:Poincare}.

We next claim that for $p\geq 4$ even, and $a,b \in \RR$ we have
\begin{align}
f_p(a,b) \geq (p-2) (a-b)^2 a^{p-2}. \label{eq:trick}
\end{align}
Assuming for the moment that \eqref{eq:trick} holds, let us prove \eqref{eq:Poincare}. Since $K_\gamma$ is positive, we have
\begin{align}
{\mathcal T} &\geq (p-2)P.V. \intint (\theta(x) - \theta(y))^2 \theta(x)^{p-2} K_\gamma(x-y) dy dx \notag\\
&\geq (p-2)  c_{d,\gamma} \intint (\theta(x) - \theta(y))^2 \theta(x)^{p-2} \frac{1}{|x-y-2\pi|^{d+\gamma}} dy dx\notag\\
&\geq \frac{(p-2) c_{d,\gamma} }{(2 \pi + |{\rm diam}({\mathbb{T}}^d)|)^{d+\gamma}} \intint (\theta(x) - \theta(y))^2 \theta(x)^{p-2} dy dx \notag\\
&=  \frac{(p-2) c_{d,\gamma} }{(2 \pi + |{\rm diam}({\mathbb{T}}^d)|)^{d+\gamma}} \intint \left( \theta^p(x) - 2 \theta^{p-1}(x) \theta(y)+ \theta^{p-2}(x) \theta^2(y) \right) dy dx\notag\\
&\geq  \frac{(p-2) c_{d,\gamma} }{(2 \pi + |{\rm diam}({\mathbb{T}}^d)|)^{d+\gamma}} \int_{{\mathbb{T}}^d} \left( \theta^p(x) |{\mathbb{T}}^d| - 2 \int_{{\mathbb{T}}^d} \theta^{p-1}(x) \theta(y) dy \right)  dx.
\label{eq:lower}
\end{align}
At this point we use that $\theta$ has zero mean, as it implies
\begin{align*}
\int_{{\mathbb{T}}^d} \theta^{p-1}(x) \theta(y) dy = 0
\end{align*}
for every $x$. It then follows from \eqref{eq:lower},  that 
\begin{align}
{\mathcal T} \geq \frac{(p-2) c_{d,\gamma}  |{\mathbb{T}}^d|}{(2 \pi + |{\rm diam}({\mathbb{T}}^d)|)^{d+\gamma}}  \| \theta\|_{L^p}^p.
\label{eq:poincare:constant}
\end{align}
This proves \eqref{eq:Poincare} with the constant 
\begin{align*}
\frac{(p-2) 2^\gamma \Gamma( (n+\gamma)/2) |{\mathbb{T}}^d|}{2 p (2 \pi + |{\rm diam}({\mathbb{T}}^d)|)^{d+\gamma}|\Gamma(-\gamma/2)| \pi^{d/2}} \geq \frac{2^\gamma \Gamma( (n+\gamma)/2) |{\mathbb{T}}^d|}{4 (2 \pi + |{\rm diam}({\mathbb{T}}^d)|)^{d+\gamma}|\Gamma(-\gamma/2)| \pi^{d/2}} = \frac{1}{C_{d,\gamma}}
\end{align*}
for any $p\geq 4$. It remains to prove the inequality \eqref{eq:trick}, which we do next.
\end{proof}

\begin{proof}[Proof of Estimate \eqref{eq:trick}]
First let $b=0$. Then \eqref{eq:trick} holds, with equality. Next, let $r = a/b$. Since $p\geq 4$ is even, checking \eqref{eq:trick} is equivalent to verifying 
\begin{align}
g_p(r):= p (r^{p-1}-1) (r-1) - 2 (r^{p/2} - 1)^2 \geq (p-2) (r-1)^2 r^{p-2} =: h_p(r).
\label{eq:check}
\end{align}
Next, note that $h_p(r) = h_p(-r)$, and that when $r>0$ we have
\begin{align*}
g_p(-r) = p(r^{p-1}+1)(r+1) - 2 (r^{p/2}-1)^2 
&= p (r^p + r^{p-1} + r + 1) - 2 (r^p - 2 r^{p/2} + 1)\notag\\
&\geq (p-2) (r^p + r^{p-2} ) \geq h_p(r) = h_p(-r).
\end{align*}
This shows that we just need to check \eqref{eq:check} for $r \geq 0$. It clearly holds at $r=0$, and also for $r \gg 1$ since it's a $p^{th}$ degree polynomial with leading coefficient $p-2 \geq 2$. Letting
\begin{align*} 
m(r) = r^{p/2-2}  + \ldots + 1 = \frac{r^{p/2-1}-1}{r-1}
\end{align*}
we can explicitly write
\begin{align}
\frac{g_p(r) - h_p(r)}{(r-1)^2} 
&= p  m(r) (r^{p/2-1}+1)  - 4  r^{p/2-1} m(r)  - 2 m(r)^2 \notag\\
& = m(r) \left( (p-4) r^{p/2-1} + p - 2 m(r) \right).\label{eq:check:2}
\end{align}
If $r \leq 1$ we are done, since $m(r) \leq m(1) = p/2 - 1$, and hence $p - 2 m(r) \geq 2$. On the other hand, if $r> 1$, we have
\begin{align} 
(p-4) r^{p/2-1} + p - 2 m(r) 
&= \frac{1}{r-1} \left( (p-4)r^{p/2} - (p-2)r^{p/2-1} + p r - (p-2)  \right) = \frac{q(r)}{r-1}.\label{eq:check:3}
\end{align}
We have $q(1)= 0$, and 
\begin{align}
2 q'(r) = p(p-4) r^{p/2-1} - (p-2)^2 r^{p/2-2} + 2p \geq 2 q'(1) = p(p-4) - (p-2)^2 + 2p = 2 (p-2) >0.
\notag
\end{align}
This proves the right side of \eqref{eq:check:3} is positive for $r \geq 1$, and thus the right side of \eqref{eq:check:2} is non-negative for $r\geq 1$ as well.
This concludes the proof of \eqref{eq:check} for all $r$, and hence of \eqref{eq:trick}.
\end{proof}

\section{Bound on the nonlinear term in Sobolev spaces}
\setcounter{equation}{0}
\label{app:nonlinear:Sobolev}

Recall the following classical commutator estimate.
Let $s >1$, $p \in (1,\infty)$, $f$ and $g$ be smooth zero-mean functions on $\TT$. Then we have the (Kenig-Ponce-Vega) commutator estimate
\begin{align}
\| \Lambda^s (f \cdot \nabla g) - f \cdot \nabla \Lambda^s g \|_{L^p} 
\leq C \left( \|\nabla f \|_{L^{p_1}} \|\Lambda^s g\|_{L^{p_2}} + \|\Lambda^s f\|_{L^{p_3}} \|\nabla g\|_{L^{p_4}} \right) \label{eq:comm:abstract}
\end{align}
where $1/p = 1/p_1 + 1/p_2 = 1/p_3 + 1/p_4$, and $p_i \in (1,\infty)$, for a sufficiently large constant $C$ that depends only on $s,p,p_i$ and the size of the periodic box. Similarly, we also make use of the fractional calculus (Kato-Ponce) inequality
\begin{align}
\|\Lambda^s (fg)\|_{L^{p}} \leq C \big(\|\Lambda^s f\|_{L^{p_1}}\|g\|_{L^{p_2}}+\|\Lambda^s g\|_{L^{p_3}}\|f\|_{L^{p_4}}\big),
\label{eq:Hs:product}
\end{align}
which is valid for sufficiently regular $f,g$, for a constant $C$ independent of $f,g$, and  for any
choice of $s \geq 0$, $1<p<\infty$,  $1<p_i \leq \infty$ and $1/p = 1/p_1 + 1/p_2 = 1/p_3 + 1/p_4$.
See, e.g. \cite{Taylor91,MuscaluSchlag13}.

For $p \in [2, \infty)$ and $f$ as above we have the Sobolev embedding
\begin{align}
\| f\|_{L^p} \leq C \|\Lambda^{1 -\frac{2}{p}} f\|_{L^2} \label{eq:sobolev:abstract}
\end{align}
for a sufficiently large constant $C$ that depends only on $p$ and the size of the periodic box.

The purpose of this appendix is to prove:
\begin{lemma}[\bf Commutator estimate]
\label{lem:commutator}
 Let $s> 1$, $\gamma \in (0,2]$, and $\Vort$ be smooth of zero-mean on $\TT$.
 Then, for any $\eps \in (0,1)$ we have
 \begin{align}
{\mathcal T} = \left| \int [\Lambda^s, u \cdot \nabla] \Vort \Lambda^s \Vort dx \right| \leq C \|\Vort\|_{H^1}^{q} + \eps \|\Vort\|_{H^{s+\gamma/2}}^{2} \label{eq:comm:sobolev}
\end{align}
where  
\begin{align*}
q = \frac{ 4 ( (4+\gamma) (s+\gamma) - 4 )}{\gamma (6 + \gamma)}
\end{align*}
for a sufficiently large constant $C$ that depends on $\eps,s,\gamma$, and the size of the box.
\end{lemma}
\begin{proof}[Proof of Lemma~\ref{lem:commutator}]
 Let $0 < \delta \ll 1$ to be chosen precisely below, and $p= 2-\delta$. The H\"older inequality and the commutator estimate \eqref{eq:comm:abstract} yield 
 \begin{align*}
{\mathcal T}
&\leq \| [\Lambda^s, u \cdot \nabla] \Vort \|_{L^p} \| \Lambda^s \Vort\|_{L^{\frac{p}{p-1}}} \notag\\
&\leq C \| \Lambda^s \Vort\|_{L^{\frac{2-\delta}{1-\delta}}} \left( \| \Vort\|_{L^{\frac{(2-\delta)(4-\delta)}{\delta}}} \| \Lambda^s \Vort\|_{L^{\frac{4-\delta}{2}}}  +  \| \Lambda^{s-1} \Vort\|_{L^{\frac{2 (2-\delta)}{\delta}}} \| \nabla \Vort\|_{L^2}  \right) 
\end{align*}
by setting $p_1= (2-\delta)(4-\delta)/\delta$, $p_2 = (4- \delta)/2 $, $p_3 = 2(2-\delta)/\delta$, and $p_4 = 2$. Using the Sobolev embedding \eqref{eq:sobolev:abstract} and the Poincar\'e inequality we obtain
\begin{align*}
{\mathcal T} 
&\leq C \| \Lambda^{s + \frac{\delta}{2-\delta}} \Vort\|_{L^2}  
		\| \Lambda^{1 - \frac{2\delta}{(2-\delta)(4-\delta)}} \Vort\|_{L^2} 
		\| \Lambda^{s - \frac{\delta}{4-\delta}} \Vort\|_{L^2}  
	+ C \| \Lambda^{s + \frac{\delta}{2-\delta}} \Vort\|_{L^2} 
		 \| \Lambda^{s - \frac{\delta}{2-\delta}} \Vort\|_{L^2}  
		 \| \Lambda \Vort\|_{L^2}  \notag\\
& \leq C \| \Lambda^{s + \frac{\delta}{2-\delta}} \Vort\|_{L^2}  
		\| \Lambda^{s - \frac{\delta}{4-\delta}} \Vort\|_{L^2}
		\| \Lambda \Vort\|_{L^2}. 		
\end{align*}
Letting $\delta$ be such that
\[ 
\frac{\delta}{2-\delta} = \frac{\gamma}{2},
\] 
and interpolating, we further bound
\begin{align*}
{\mathcal T}
&\leq C \| \Lambda^{s + \frac{\gamma}{2}} \Vort\|_{L^2}^{2 - \alpha} \| \Lambda \Vort\|_{L^2}^{1+\alpha} 
\end{align*}
for a positive constant $C$ that depends on $\gamma, s$, and the size of the domain, where 
\begin{align*}
\alpha = \frac{ \gamma (6+\gamma)}{( 4 + \gamma)(2s- 2 + \gamma)} .
\end{align*}
We conclude the proof of the lemma with the $\eps$-Young inequality, and letting $q =  2(1+\alpha)/\alpha$.
\end{proof}

\section*{Acknowledgements}
The work of PC was supported in part by NSF grants DMS-1209394, DMS-1265132,  and DMS-1240743.
NGH gratefully acknowledges the support of the Institute for Mathematics and its Applications (IMA) at the University of Minnesota. 
The work of VV was supported in part by the NSF grant DMS-1211828.

\newcommand{\etalchar}[1]{$^{#1}$}

\end{document}